[1994/12/01]% LaTeX date must December 1994 or later
\documentclass[manuscript, a4paper, reqno]{aomart}

\chardef\bslash=`\\ 

\newtheorem{thm}{Theorem}[section]
\newtheorem{cor}[thm]{Corollary}
\newtheorem{lem}[thm]{Lemma}

\newtheorem{prop}[thm]{Proposition}

\theoremstyle{remark} 

\usepackage{dsfont}
\usepackage{amssymb}

\usepackage{graphicx}
\usepackage{moresize}

\usepackage[numeric]{amsrefs}

\usepackage[english]{babel}

\theoremstyle{definition}
\newtheorem{defn}{\textsc{Definition}}[section]
\newtheorem{rem}{Remark}[section]
\newtheorem*[{}\it]{notation}{Notation}

\newtheorem*[{}\it]{rest}{\textsc{Theorem}}
\newtheorem*[{}\it]{quest}{\textsc{Question}}
\newtheorem*[{}\it]{projone}{\texttt{Project 1 (Future Work)}}
\newtheorem*[{}\it]{projtwo}{\texttt{Project 2 (Current Work)}}
\newtheorem*[{}\it]{projthree}{\texttt{Project 3 (Current Work)}}
\newtheorem*[{}\it]{projfour}{\texttt{Project 4 (Future Work)}}
\newtheorem*[{}\it]{projfive}{\texttt{Project 5 (Future Work)}}
\newtheorem*[{}\it]{proofoflemma}{Proof of Lemma}

\usepackage[textwidth=6in,textheight=9.25in, paperwidth=7.5in,paperheight=11.25in, inner=2.5cm, outer=2.5cm]{geometry} % add in the following: inner=2.0cm, outer=1.5cm

\usepackage[usenames,dvipsnames]{xcolor}

\usepackage{multicol}

\usepackage{mathrsfs}

\usepackage{txfonts}

\usepackage{enumerate}

\usepackage{tikz}
\usetikzlibrary{shapes.geometric, arrows}

\usepackage{amssymb}
\usepackage[misc]{ifsym}

\title[]{On Global-in-time Chaotic Weak Solutions of the Liouville Equation for Hard Spheres}
\author[]{Mark Wilkinson}\thanks{Department of Mathematics and Computer Science (MACS), Heriot-Watt University, Edinburgh, EH14 4AS. (\Letter) \href{mailto:mark.wilkinson@hw.ac.uk}{mark.wilkinson@hw.ac.uk}}

\allowdisplaybreaks

\newcommand{\ov}{\overline}

\hypersetup{citecolor=red}

\usepackage{tikz}
\usetikzlibrary{patterns}
\usepackage{tkz-graph}
\tikzstyle{vertex}=[circle, draw, inner sep=0pt, minimum size=5pt]

\usepackage{wasysym}

\newcommand{\mres}{\mathbin{\vrule height 1.6ex depth 0pt width
0.13ex\vrule height 0.13ex depth 0pt width 1.3ex}}

\definecolor{ccqqqq}{rgb}{0.8,0,0}\definecolor{ffvvqq}{rgb}{1,0.3333333333333333,0}\definecolor{wwwwww}{rgb}{0.4,0.4,0.4}

\usetikzlibrary{arrows}
\definecolor{cqcqcq}{rgb}{0.7529411764705882,0.7529411764705882,0.7529411764705882}
\definecolor{ffqqqq}{rgb}{1,0,0}\definecolor{uququq}{rgb}{0.25098039215686274,0.25098039215686274,0.25098039215686274}\definecolor{wqwqwq}{rgb}{0.3764705882352941,0.3764705882352941,0.3764705882352941}\definecolor{zzccff}{rgb}{0.6,0.8,1}

\begin{document}

\maketitle
\begin{abstract}
\noindent We outline a new method of construction of global-in-time weak solutions of the Liouville equation -- and also of the associated BBGKY hierarchy -- corresponding to the hard sphere singular Hamiltonian. Our method makes use only of geometric reflection arguments on phase space. As a consequence of our method, in the case of $N=2$ hard spheres, we show for {\em any} chaotic initial data, the unique global-in-time weak solution $F$ of the Liouville equation is realised as $F=\mathsf{R}(f\otimes f)$ in the sense of tempered distributions on $T\mathbb{R}^{6}\times (-\infty, \infty)$, where $\mathsf{R}:\mathscr{S}\rightarrow\mathscr{S}'$ is a `reflection-type operator' on Schwartz space, and $f$ is a global-in-time classical solution of the 1-particle free transport Liouville equation on $T\mathbb{R}^{3}$.
%We claim that the work herein fills a gap in the monograph of \textsc{Gallagher, Saint-Raymond and Texier}, in which the derivation of the BBGKY hierarchy from the underlying hard sphere ODEs was performed erroneously. 
\end{abstract}

\section{Introduction}
In this article, we employ a new method of construction of global-in-time weak solutions of the Liouville equation for physical hard sphere dynamics given by
\begin{equation}\label{liouenn}
\frac{\partial F_{N}}{\partial t}+(V\cdot \nabla_{X})F_{N}=0 \quad \text{on}\hspace{2mm}\mathcal{D}_{N}\times (-\infty, \infty),
\end{equation}
which is subject to mass transport boundary conditions. The set $\mathcal{D}_{N}$ denotes the phase space bundle $\sqcup_{X\in\mathcal{P}_{N}}\mathcal{V}(X)$, where $\mathcal{P}_{N}\subset\mathbb{R}^{3N}$ is the {\em hard sphere table} given by
\begin{equation*}
\mathcal{P}_{N}:=\left\{
X=[x_{1}, ..., x_{N}]\in\mathbb{R}^{3N}\,:\,|x_{i}-x_{j}|\geq \varepsilon\hspace{2mm}\text{for}\hspace{2mm}i\neq j
\right\}
\end{equation*}
for $N\geq 2$ congruent hard spheres of radius $\varepsilon>0$, and $\mathcal{V}(X)\subseteq\mathbb{R}^{3N}$ is the (generalised) tangent space to $\mathcal{P}_{N}$ at $X\in\mathcal{P}_{N}$. Our method makes uses neither of the hard sphere flow on $\mathcal{P}_{N}$, nor of a priori estimates or weak compactness methods. Moreover, our method reveals what we believe to be new information about the the nature of propagation of chaos for hard sphere systems. Indeed, we claim that if one furnishes equation \eqref{liouenn} with initial data $F_{N, 0}$ of the shape
\begin{equation*}
F_{N, 0}(X, V):=\prod_{j=1}^{N}\phi_{0}(x_{j}, v_{j})
\end{equation*}
for some initial 1-particle distribution function $\phi_{0}$, then the associated unique global-in-time weak solution $F_{N}$ of the Liouville equation \eqref{liouenn} admits the structure
\begin{equation}\label{chaos}
F_{N}=\mathsf{R}_{N}\left(\bigotimes_{j=1}^{N}\phi\circ\pi_{j}\right)\quad \text{in}\hspace{2mm}\mathscr{S}',
\end{equation}
where $\mathscr{S}'$ denotes the space of tempered distributions on $T\mathbb{R}^{3N}\times (\infty, \infty)$, $\pi_{j}:T\mathbb{R}^{3N}\rightarrow T\mathbb{R}^{3}$ is a canonical projection operator for $j=1, ..., N$, the symbol $\otimes$ denotes the standard tensor product operator on Schwartz functions defined on $T\mathbb{R}^{3}$, and $\mathsf{R}_{N}:\mathscr{S}\rightarrow\mathscr{S}'$ is a {\em reflection-type} operator on Schwartz space $\mathscr{S}$ which is built using classical Boltzmann scattering matrices. We term the identity \eqref{chaos} in this article propagation of {\bf Schwartz chaos}. Note that this is a {\em global-in-time} statement about the chaotic structure of weak solutions, in that the equality \eqref{chaos} holds in the sense of tempered distributions on the whole set $\mathcal{D}_{N}\times (\infty, \infty)$ and not a strict subset thereof. In this work, we detail only the case of two hard spheres in the whole space $\mathbb{R}^{3}$, i.e. we consider only $N=2$. However, our method can be extended to an $N$-particle system. We tackle this in forthcoming work \textsc{Wilkinson} \cite{me2}.
\subsection{Set-up of the Problem}
Let us now set up in more detail the problem of focus in this article. In all that follows, we fix $N=2$, and for notational convenience we drop it as a subscript from all objects which depend on it. We also suppose that the radius $\varepsilon>0$ of the congruent hard spheres is given and fixed. We also suppress notational dependence of all objects on $\varepsilon$. 

Let $X=[x, \ov{x}]\in\mathbb{R}^{6}$ denote\footnote{In general, in this article, square brackets will be used to denote the vector concatenation of vector data.} the concatenation of the centres of mass $x, \ov{x}\in\mathbb{R}^{3}$ of the two hard spheres, and let also $V=[v, \ov{v}]\in\mathbb{R}^{6}$ denote the concatenation of their associated linear velocities $v, \ov{v}\in\mathbb{R}^{6}$. The hard sphere singular Hamiltonian $H_{\ast}:T\mathbb{R}^{6}\rightarrow [0, \infty]$ is defined by
\begin{equation*}
H_{\ast}(X, V):=\frac{1}{2}(|v|^{2}+|\ov{v}|^{2})+\left\{
\begin{array}{ll}
0 & \quad \text{if} \hspace{2mm}|x-\ov{x}|\geq \varepsilon\vspace{2mm}\\
\infty & \quad \text{if}\hspace{2mm}\text{otherwise}.
\end{array}
\right.
\end{equation*}
In the kinetic theory literature, the Liouville equation associated to $H_{\ast}$ is typically written in the form
\begin{equation}\label{lioutwo}
\frac{\partial F}{\partial t}+(V\cdot\nabla_{X})F=0 \quad\text{on}\hspace{2mm}\mathcal{D}\times (-\infty, \infty),
\end{equation}
where $\mathcal{D}$ is the {\em hard sphere phase space bundle} given by
\begin{equation*}
\mathcal{D}:=\bigsqcup_{X\in\mathcal{P}}\mathcal{V}(X),
\end{equation*}
and $\mathcal{P}\subset\mathbb{R}^{6}$ is the {\em hard sphere table} defined by
\begin{equation*}
\mathcal{P}:=\left\{
X=[x, \ov{x}]\in\mathbb{R}^{6}\,:\,|x-\ov{x}|\geq \varepsilon
\right\},
\end{equation*}
which admits the structure of a real analytic manifold with boundary. Indeed, the boundary $\partial\mathcal{P}$ of the hard sphere table is diffeomorphic to the product manifold $\mathbb{S}^{2}\times\mathbb{R}^{3}$. Moreover, the (generalised) tangent space $\mathcal{V}(X)$ to the manifold with boundary $\mathcal{P}$ at the point $X\in\mathcal{P}$ is given explicitly by
\begin{equation*}
\mathcal{V}(X):=\left\{
\begin{array}{ll}
\mathbb{R}^{6} & \quad \text{if}\hspace{2mm}X\in\mathcal{P}^{\circ}, \vspace{2mm}\\
\Sigma_{X} & \quad \text{if}\hspace{2mm}X\in\partial\mathcal{P},
\end{array}
\right.
\end{equation*}
where $\mathcal{P}^{\circ}$ denotes the interior of the (closed) table $\mathcal{P}$, and $\Sigma_{X}\subset\mathbb{R}^{6}$ denotes the closed half-space 
\begin{equation*}
\Sigma_{X}:=\left\{
V\in\mathbb{R}^{6}\,:\,V\cdot\widehat{\nu}_{X}\leq 0
\right\},
\end{equation*}
for the unit-norm vector $\widehat{\nu}_{X}\in\mathbb{R}^{6}$ given by
\begin{equation*}
\widehat{\nu}_{X}:=\frac{1}{\sqrt{2}\varepsilon}\left[
\begin{array}{c}
\ov{x}-x\\
x-\ov{x}
\end{array}
\right]\in\mathbb{S}^{5}.
\end{equation*} 
Although the Liouville equation is written as \eqref{lioutwo} above, this is in many respects misleading to those who seek the `correct' weak formulation of the transport equation associated to the singular Hamiltonian $H_{\ast}$. Firstly, the transport equation is underdetermined in the sense that boundary conditions on $\partial\mathcal{D}$ have not yet been prescribed. Secondly, a probability distribution function $F$ can be expected to satisfy \eqref{lioutwo} in the sense of distributions on $\mathcal{D}\times (-\infty, \infty)$ no more than the trajectory $t\mapsto Z(t):=[X(t), V(t)]$ can be expected to satisfy the ODE system
\begin{equation}\label{newton}
\frac{dZ}{dt}=JZ
\end{equation}
in the sense of $\mathbb{R}^{12}$-valued distributions on $\mathbb{R}$, where $J\in\mathbb{R}^{12\times 12}$ is the block matrix
\begin{equation*}
J:=\left(
\begin{array}{cc}
0 & I \\
0 & 0
\end{array}
\right),
\end{equation*}
with $0\in\mathbb{R}^{6\times 6}$ the zero matrix and $I\in\mathbb{R}^{6\times 6}$ the identity matrix. Indeed, in general, one ought to interpret Newton's equations of motion \eqref{newton} as
\begin{equation*}
\frac{dZ}{dt}=\mu \quad \text{in}\hspace{2mm}\mathscr{D}'(\mathbb{R}, \mathbb{R}^{12}),
\end{equation*}
where $\mu$ is an $\mathbb{R}^{12}$-valued Radon measure on $\mathbb{R}$ (dependent on initial data) which is determined by the conservation of linear momentum, angular momentum and kinetic energy, and $\mathscr{D}'(\mathbb{R}, \mathbb{R}^{12})$ denotes the space of $\mathbb{R}^{12}$-valued distributions on $\mathbb{R}$. In this same vein, we claim that the correct formulation of the Liouville equation associated to the singular Hamiltonian $H_{\ast}$ is
\begin{equation}\label{weakdude}
\frac{\partial F}{\partial t}+(V\cdot\nabla_{X})F=C[F] \tag{L}
\end{equation}
where $C$ is the collision operator defined on suitable functions on $\mathcal{D}$ by
\begin{equation}\label{collop}
C[F]:=\int_{\partial\mathcal{P}}\int_{\mathbb{R}^{6}}F(Y, V)V\cdot\widehat{\nu}_{Y}\,dVd\mathscr{H}(Y),
\end{equation}
with $\mathscr{H}$ denoting the Hausdorff measure on $\partial\mathcal{P}$. The equality in \eqref{weakdude} is to be interpreted in the sense of distributions on $\mathcal{D}\times (-\infty, \infty)$. The (reasonably lengthy) derivation of this equation will be performed in full detail in section \ref{derivation} below. We state the precise definition of global-in-time weak solution of \eqref{weakdude} in section \ref{mainres} below. In addition, the BBGKY hierarchy associated to this equation is given by
\begin{equation}\label{beebeegeekay}
\left\{
\begin{array}{l}
\displaystyle \frac{\partial F^{(1)}}{\partial t}+(v\cdot\nabla_{x})F^{(1)}=C^{(1)}[F^{(2)}] \vspace{2mm}\\
\displaystyle \frac{\partial F^{(2)}}{\partial t}+(V\cdot\nabla_{X})F^{(2)}=C^{(2)}[F^{(2)}],
\end{array}
\right. \tag{BBGKY}
\end{equation}
where $C^{(1)}$ is a collision operator of the shape
\begin{equation*}
\begin{array}{c}
C^{(1)}[F]:= \\
\displaystyle \frac{1}{\sqrt{2}}\int_{\mathbb{R}^{3}}\int_{\mathbb{S}^{2}}\int_{\mathcal{C}^{+}(n; v)}\left(F(y, v, y+\varepsilon n, \ov{v})-F(y, v_{n}', y+\varepsilon n, \ov{v}_{n}')\right)\beta(v-\ov{v}, n)\,d\ov{v}dndy,
\end{array}
\end{equation*}
and where $v_{n}':=v-((v-\ov{v})\cdot n)n$, $\ov{v}_{n}':=\ov{v}+((v-\ov{v})\cdot n)n$, the map $\beta$ is a collision cross-section and $\mathcal{C}^{+}(n; v)\subset\mathbb{R}^{3}$ is a half-space of velocities to be defined in the sequel. Finally, the collision operator $C^{(2)}$ is simply $C$ given by \eqref{collop} above. 

Before we present our new method of construction and main results, let us mention briefly some of the major results in the literature which (i) tackle the construction of various notions of solution of the Liouville equation \eqref{weakdude}, (ii) tackle construction of some notion of solution of its associated BBGKY hierachy, and (iii) investigate the (in)ability of these equations to preserve the structure of so-called {\em chaotic} initial data.
\subsection{On Some Previously-established Results}
We shall not offer the reader a comprehensive overview of the literature on the Liouville equation for hard spheres and the hierarchies associated to it here, as this has been attempted many times before elsewhere. Indeed, we direct the reader to the thesis of \textsc{Simonella} \cite{serg}, or the monograph of \textsc{Gallagher, Saint-Raymond and Texier} \cite{gallagher2014newton} for a more systematic mapping of the literature. 

\subsubsection{Mild Solutions}
When considering the Liouville equation associated to Newton's equations of motion for hard spheres, it is typical to consider {\em mild} solutions thereof (as opposed to {\em distributional solutions} thereof). Indeed, using the terminology of \textsc{Cercignani, Illner and Pulvirenti} (\cite{cercignani2013mathematical}, page 69), we say a map $F_{N}$ is a mild solution of the Liouville equation for $N$ hard spheres corresponding to the initial datum $F_{N, 0}$ if and only if
\begin{equation*}
F_{N}(Z, t)=F_{N, 0}(T^{(N)}_{-t}Z),
\end{equation*}
where $\{T^{(N)}_{t}\}_{t\in\mathbb{R}}$ denotes the hard sphere flow defined (($\mathscr{L}_{3N}\mres\mathcal{D}_{N})\otimes\mathscr{L}_{3N}$-almost everywhere) on $\mathcal{D}_{N}$. With this notion of solution in place, one can in turn use mild solutions of the Liouville equation to write down what one means by a solution of its associated hierarchy. This is the approach taken in \textsc{Gallagher, Saint-Raymond and Texier} (\cite{gallagher2014newton}, chapter 4), amongst other places. The note of \textsc{Spohn} \cite{spohn2006integrated} also outlines an approach by which one can derive a mild formulation of the BBGKY hierarchy for hard spheres. 

To the knowledge of the author, it has not been demonstrated in the literature that mild solutions of the Liouville equation are in any sense equivalent to weak solutions thereof. This seems like a small distinction, however its establishment makes the derivation of the hierarchy entirely straightforward. Indeed, in the final section \ref{bbgkysection} of this article, we show that the derivation of the BBGKY hierarchy associated to the Liouville equation is almost immediate once one has a construction of weak solutions of \eqref{weakdude}. 
\subsubsection{Chaotic Solutions}
It is known that the dynamics of the hiearchy (in whichever sense is appropriate) does not preserve the structure of {\em chaotic} initial data for all times. See, for instance, the recent work of \textsc{Denlinger} \cite{denlinger2017structure}. Recently, \textsc{Pulvirenti and Simonella} \cite{pulvirenti2017boltzmann} have performed a detailed study of so-called {\em correlation errors} which measure the manner in which solutions of the hierarchy deviate from profiles which are in product form. In this article, we study the chaotic structure of weak solutions of the Liouville equation not at the level of functions on $\mathcal{D}\times (-\infty, \infty)$, but rather at the level of distributions on $\mathcal{D}\times (-\infty, \infty)$. We claim, but do not offer further explanation here, that this perspective is useful in the pursuit of the global-in-time Boltzmann-Grad limit of weak solutions of the BBGKY hierarchy.
\subsubsection{Contributions of this Article}
To the knowledge of the author, there does not exist a comprehensive existence and uniqueness theory of {\em global-in-time weak solutions} of the Liouville equation \eqref{weakdude} for hard spheres, even in the case of only $N=2$ hard spheres. The present article only tackles the case of two particles explicitly, but the methods introduced herein to do so may be extended to the case of three or more spheres. Let us now introduce the main intuitive idea behind our approach in this article.
\subsection{A Motivating Example: The Sinai Billiard}
Overall, our aim is to study hard sphere dynamics by eliminating collisions entirely. We illustrate the basic idea behind our approach to the hard sphere table $\mathcal{P}$ in this article by considering the well-studied {\bf Sinai billiard} in $\mathbb{R}^{2}$ (which is also known as the {\em Lorenz gas}). We invite the reader to the book of \textsc{Tabachnikov} \cite{tabachnikov2005geometry} for an introduction to the theory of abstract billiards. The Sinai billiard is a model for the kinetic energy-conserving trajectory $t\mapsto x(t)$ of a point particle governed by the evolution equations given formally by
\begin{equation}\label{pisheasy}
\frac{dx}{dt}=v\quad\text{and}\quad\frac{dv}{dt}=0,
\end{equation}
where $t\mapsto x(t)$ is constrained to evolve in the {\em Sinai billiard table} $\Omega\subset\mathbb{R}^{2}$ defined to be
\begin{equation*}
\Omega:=\left\{
y\in \mathbb{R}^{2}\,:\, |y_{i}|\leq 1\hspace{2mm}\text{for}\hspace{2mm}i=1, 2, \hspace{2mm}\text{and}\hspace{2mm}|y|\geq r
\right\},
\end{equation*}
with $1>r>0$ denoting the radius of the fixed {\em scatterer} at the centre of the box. In this article, we understand $[x, v]$ to be an energy-conserving global-in-time weak solution of the above equations \eqref{pisheasy} if and only if $x\in C^{0}((\infty, \infty), \Omega)$ and $v\in\mathrm{BV}_{\mathrm{loc}}((-\infty, \infty), \mathbb{R}^{2})$ satisfy the equations
\begin{equation*}
\frac{dx}{dt}=v \quad \text{and}\quad\frac{dv}{dt}=\mu_{(x_{0}, v_{0})}
\end{equation*}
in the sense of $\mathbb{R}^{2}$-valued distributions on $(-\infty, \infty)$, where $\mu_{(x_{0}, v_{0})}$ is an atomic measure on $(-\infty, \infty)$ supported on the set of {\em collision times} of the trajectory $x\mapsto x(t)$. In all that follows, we are not interested in proving qualitative properties of trajectories governed by \eqref{pisheasy} (such as ergodicity or recurrence, for example), but are interested {\em only} in the construction of weak solutions thereof. 
\subsubsection{Construction of Solutions}
From the point of view of the mathematical analyst, the available methods of construction of global-in-time weak solutions $[x, v]$ of \eqref{pisheasy} are few. One intuitive method by which to construct solutions is the so-called {\em method of surgery}. Indeed, by the `bending' of a straight line trajectory in $\mathbb{R}^{2}$ whenever it hits the boundary $\partial\Omega$ of the billiard table $\Omega$, for any $x_{0}\in \Omega$ and suitable $v_{0}\in\mathbb{R}^{2}$, it seems `obvious' that there exists a unique global-in-time weak solution of system \eqref{pisheasy}. Characterisation of the set of $[x_{0}, v_{0}]$-dependent collision times, namely
\begin{equation*}
\left\{t\in(-\infty, \infty)\,:\,x(t; [x_{0}, v_{0}])\in\partial\Omega\right\}
\end{equation*} 
given an initial datum $[x_{0}, v_{0}]\in\Omega\times\mathbb{R}^{2}$ is, however, not an easy task. As such, this method of construction of weak solutions is somewhat cumbersome.

Another approach, to be employed later in this article in the case of hard spheres, is to consider a construction of dynamics associated to \eqref{pisheasy} by way of an identification procedure on {\em two copies} of the Sinai billiard table $\Omega$ (see \textsc{Tabachnikov} \cite{tabachnikov2005geometry}, chapter 1). Beginning with the component $\Gamma_{1}$ of the boundary $\partial\Omega$ given by
\begin{equation*}
\Gamma_{1}:=\left\{
y\in\mathbb{R}^{2}\,:\, |y_{i}|=1\hspace{2mm}\text{for some}\hspace{2mm}i=1, 2
\right\}
\end{equation*}
and identifying  
\begin{equation*}
(-1, y_{2})\quad \text{with}\quad (1, y_{2}) \quad \text{for}\hspace{2mm}y_{2}\in[0, 1],
\end{equation*}
in addition to
\begin{equation*}
(y_{1}, -1)\quad \text{with}\quad (y_{1}, 1) \quad \text{for}\hspace{2mm}y_{1}\in[0, 1],
\end{equation*}
the resulting quotient space of each copy is homeomorphic to a genus 1 surface in $\mathbb{R}^{3}$ with boundary that is homeomorphic to $\mathbb{S}^{1}$. In turn, by considering the component $\Gamma_{2}$ of the boundary $\partial\Omega$ given by
\begin{equation*}
\Gamma_{2}:=\left\{
y\in\mathbb{R}^{2}\,:\,|y|=r
\right\}
\end{equation*}
and identifying 
\begin{equation*}
(y_{1}, y_{2})\in\Gamma_{2}\hspace{2mm}\text{on copy 1 with}\hspace{2mm} (-y_{1}, -y_{2})\hspace{2mm}\text{on copy 2}
\end{equation*}
along with
\begin{equation*}
(y_{1}, y_{2})\in\Gamma_{2}\hspace{2mm}\text{on copy 2 with}\hspace{2mm} (-y_{1}, -y_{2})\hspace{2mm}\text{on copy 1}
\end{equation*}
one produces a quotient space which is homeomorphic to a genus 2 surface $\mathcal{M}$ in $\mathbb{R}^{3}$. Moreover, subject to endowing this quotient space with a differentiable structure and its tangent bundle $T\mathcal{M}$ with a symplectic form, one can view the piecewise linear trajectories $t\mapsto x(t)$ on $\Omega$ as integral curves of a complete vector field on the tangent bundle adjunction manifold $\mathcal{M}$. However, rigorous works which faithfully map classical solutions on the auxiliary table $\mathcal{M}$ to weak solutions of the system \eqref{pisheasy} seem to be absent in the literature.

As is typical in the theory of dynamical billiards, there is more than one way by which one may build a suitable quotient space from a starting table. For instance, if the initial table is a square in $\mathbb{R}^{2}$, one may produce either a flat torus (an example of a compact auxiliary table), or the whole plane $\mathbb{R}^{2}$ (an example of a non-compact auxiliary table). With the case of the hard sphere table $\mathcal{P}$ in mind, it is more convenient to work with an identification procedure which does not lead to a compact phase space. We outline this now.
\subsubsection{Copying and Identifying the Sinai Billiard Table}
Instead of building an auxiliary table from only two copies of $\Omega$, one can build an auxiliary table that admits the structure of a smooth 2-dimensional adjunction manifold $\mathcal{N}$ by identifying countably-infinitely many copies of $\Omega$ with one another in a suitable manner. Rather than write out this procedure carefully, we simply refer the reader to Figure 1 above in which this formal procedure is articulated diagramatically. 

The major benefit in studying the evolution of a point particle on the adjunction manifold $\mathcal{N}$, as opposed to one on the Sinai table $\Omega$, is that the study of piecewise smooth dynamics becomes one of {\em smooth} dynamics. In particular, the Liouville equation associated to \eqref{pisheasy} takes the form of a {\em free transport equation} on the adjunction manifold tangent bundle $T\mathcal{N}$. Of course, the main challenge thereafter is in the faithful mapping of smooth solutions of the `Liouville equation' posed on $T\mathcal{N}$ to the Liouville equation associated to \eqref{pisheasy}. This is the problem we tackle for the hard sphere table $\mathcal{P}$ in the sequel.
\begin{figure}
\centering
\includegraphics[scale=0.05]{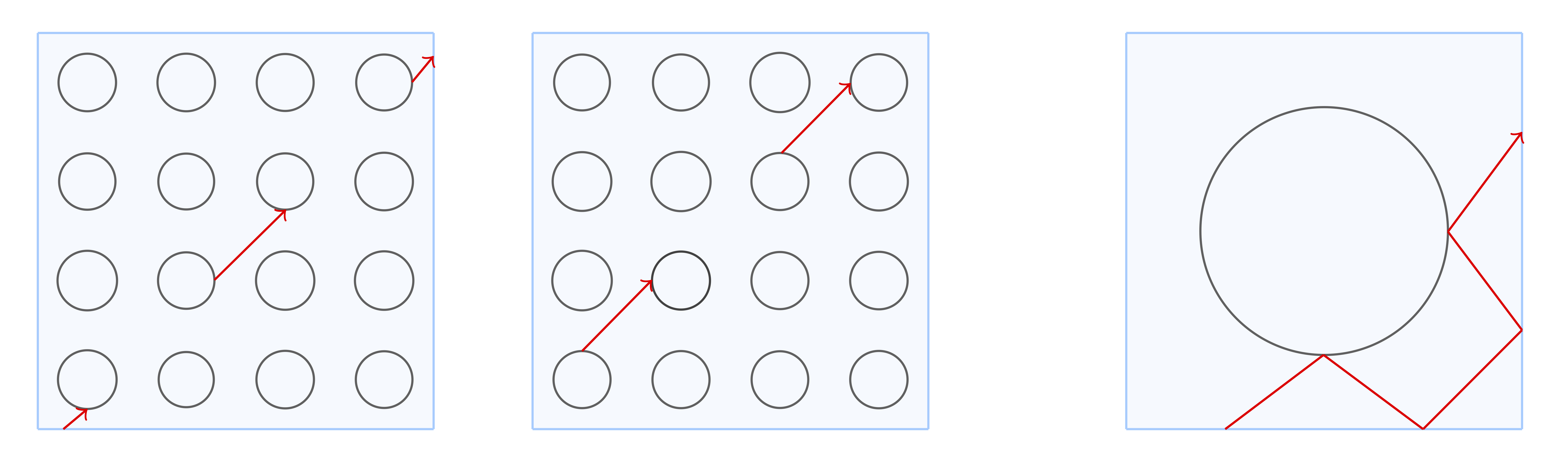}
\caption{The straight line trajectory of a point particle on the adjunction manifold $\mathcal{N}$ constructed from the periodic Lorenz gas table, and the corresponding trajectory with collisions on the Sinai table $\Omega\subset\mathbb{R}^{2}$.}
\end{figure}
\subsection{Main Results}\label{mainres}
In this section we present the main results of this article. In order that we may speak meaningfully of the objects we claim to construct, we begin with some basic definitions.
\subsubsection{Basic Definitions}
We recall the following common notion.
\begin{defn}[Symmetric Functions on $\mathcal{D}$]
Suppose $N\geq 2$. Let $\mathfrak{s}(N)$ denote the symmetric group on $N$ letters, whose elements $s\in\mathfrak{s}(N)$ admit the following natural action on $\mathcal{D}$:
\begin{equation*}
sZ:=[(x_{s(1)}, v_{s(1)}), ..., (x_{s(N)}, v_{s(N)})] \quad \text{for}\hspace{2mm}Z=[(x_{1}, v_{1}), ..., (x_{N}, v_{N})]\in\mathcal{D}.
\end{equation*}
We say a function $G:\mathcal{D}\rightarrow\mathbb{R}$ is {\bf symmetric} if and only if
\begin{equation*}
G(sZ)=G(Z)\quad \text{for all}\hspace{2mm}Z\in\mathcal{D}.
\end{equation*}
\end{defn}
Of course, from a modelling perspective, a probability distribution function $F_{0}$ on $\mathcal{D}$ which is symmetric is intended to model a system of {\em indistinguishable} particles. However, our notion of solution of the Liouville equation associated to $H_{\ast}$ does not require symmetry of its initial data. Indeed, we now define what we mean by a weak solution of the Liouville equation associated to the singular Hamiltonian.
\begin{defn}[Global-in-time Weak Solution of the Liouville Equation \eqref{weakdude}]\label{weaksoldef}
Suppose an initial distribution $F_{0}\in C^{0}(\mathcal{D})\cap L^{1}(\mathcal{D})$ satisfying 
\begin{equation*}
\int_{\mathcal{P}}\int_{\mathbb{R}^{6}}F_{0}(X, V)\,dVdX=1\quad \text{and}\quad F_{0}(X, V)\geq 0
\end{equation*}
is given. We say that $F\in C^{0}((-\infty, \infty), L^{1}(\mathcal{D}))$ is a physical {\bf global-in-time weak solution} of the Liouville equation \eqref{weakdude} if and only if
\begin{equation*}
\begin{array}{c}
\displaystyle \int_{\mathcal{P}}\int_{\mathbb{R}^{6}}\int_{-\infty}^{\infty}\left(\partial_{t}\Phi(X, V, t)+V\cdot\nabla_{X}\Phi(X, V, t)\right)F(X, V, t)\,dtdVdX= \vspace{2mm} \\
\displaystyle -\int_{\partial\mathcal{P}}\int_{\mathbb{R}^{6}}\int_{-\infty}^{\infty}F(Y, V, t)\Phi(Y, V, t)V\cdot\widehat{\nu}_{Y}\,dtdVd\mathscr{H}
\end{array}
\end{equation*}
holds true for all test functions $\Phi\in C^{1}_{c}(T\mathbb{R}^{6})$. Moreover, $F$ obeys the conservation of mass
\begin{equation*}
\int_{\mathcal{P}}\int_{\mathbb{R}^{6}}F(X, V, t)\,dVdX=1,
\end{equation*}
the conservation of linear momentum
\begin{equation*}
\int_{\mathcal{P}}\int_{\mathbb{R}^{6}}(v+\ov{v})F(X, V, t)\,dVdX=\int_{\mathcal{P}}\int_{\mathbb{R}^{6}}(v+\ov{v})F_{0}(X, V)\,dVdX,
\end{equation*}
the conservation of angular momentum
\begin{equation*}
\int_{\mathcal{P}}\int_{\mathbb{R}^{6}}(x\wedge v+\ov{x}\wedge\ov{v})F(X, V, t)\,dVdX=\int_{\mathcal{P}}\int_{\mathbb{R}^{6}}(x\wedge v+\ov{x}\wedge\ov{v})F_{0}(X, V)\,dVdX,
\end{equation*}
and the conservation of kinetic energy
\begin{equation*}
\frac{1}{2}\int_{\mathcal{P}}\int_{\mathbb{R}^{6}}|V|^{2}F(X, V, t)\,dVdX=\frac{1}{2}\int_{\mathcal{P}}\int_{\mathbb{R}^{6}}|V|^{2}F_{0}(X, V)\,dVdX
\end{equation*}
for all times $t\in (-\infty, \infty)$. Finally, $F(X, V, 0)=F_{0}(X, V)$ for all $[X, V]\in\mathcal{D}$.
\end{defn}
We now set out what we mean by a weak solution of the BBGKY hierarchy associated to the Liouville equation \eqref{weakdude}.
\begin{defn}[Global-in-time Weak Solution of the Hierarchy]
Suppose an initial datum $G_{0}\in C^{1}(\mathcal{D})\cap L^{1}(\mathcal{D})$ satisfying the symmetry
\begin{equation*}
G_{0}(sX, sV)=G_{0}(X, V)\quad \text{for all}\hspace{2mm}[X, V]\in\mathcal{D}
\end{equation*}
and each $s\in\mathfrak{s}(2)$ is given. We say that the pair of maps $(G^{(1)}, G^{(2)})$ with membership
\begin{equation*}
G^{(1)}\in C^{0}((-\infty, \infty), L^{1}(T\mathbb{R}^{3}))
\end{equation*}
and
\begin{equation*}
G^{(2)}\in C^{0}((-\infty, \infty), L^{1}(\mathcal{D}))
\end{equation*}
is a {\bf global-in-time weak solution} of the BBGKY hierarchy associated to the initial data
\begin{equation*}
G_{0}^{(1)}(x, v):=\int_{\mathbb{R}^{3}(x)}\int_{\mathbb{R}^{3}}G_{0}(X, V)\,d\ov{x}d\ov{v}\quad \text{for all}\hspace{2mm}[x, v]\in T\mathbb{R}^{3}
\end{equation*}
and
\begin{equation*}
G_{0}^{(2)}(X, V):=G_{0}(X, V)\quad \text{for all}\hspace{2mm}[X, V]\in\mathcal{D}
\end{equation*}
if and only if
\begin{equation*}
\begin{array}{c}
\displaystyle \int_{\mathbb{R}^{3}}\int_{\mathbb{R}^{3}}\int_{-\infty}^{\infty}\left(\partial_{t}+v\cdot\nabla_{x}\right)\psi(x, v, t)G^{(1)}(x, v, t)\,dtdvdx \vspace{2mm}\\
=\displaystyle -\frac{1}{\sqrt{2}}\int_{\mathbb{R}^{3}}\int_{\mathbb{S}^{2}}\int_{\mathcal{C}^{-}(n)}\int_{-\infty}^{\infty}\beta(v-\ov{v}, n)\psi(x, v, t)\bigg(G^{(2)}(x, v_{n}', x+\varepsilon n, \ov{v}_{n}', t) \vspace{2mm} \\
-G^{(2)}(x, v, x+\varepsilon n, \ov{v}, t)\bigg)\,dtdvd\ov{v}dndx
\end{array}
\end{equation*}
and
\begin{equation*}
\begin{array}{c}
\displaystyle \int_{\mathcal{P}}\int_{\mathbb{R}^{6}}\int_{-\infty}^{\infty}\left(
\partial_{t}+V\cdot\nabla_{X}
\right)\Psi(X, V, t)G^{(2)}(X, V, t)\,dtdVdX \vspace{2mm}\\
=\displaystyle -\int_{\partial\mathcal{P}}\int_{\mathbb{R}^{6}}\int_{-\infty}^{\infty}G^{(2)}(Y, V, t)\Psi(Y, V, t)V\cdot\widehat{\nu}(Y)\,dtdVdY
\end{array}
\end{equation*}
for all test functions $\psi\in C^{\infty}_{c}(T\mathbb{R}^{3}\times (-\infty, \infty))$ and $\Psi\in C^{\infty}_{c}(T\mathbb{R}^{6}\times(-\infty, \infty))$, where the set $\mathcal{C}^{-}(n)\subset\mathbb{R}^{6}$ is given by
\begin{equation*}
\mathcal{C}^{-}(n):=\left\{
V\in\mathbb{R}^{6}\,:\,V\cdot\frac{1}{\sqrt{2}}\left[
\begin{array}{c}
n \\
-n
\end{array}
\right]\geq 0
\right\}
\end{equation*}
and $\beta:\mathbb{R}^{3}\times\mathbb{S}^{2}\rightarrow\mathbb{R}$ denotes the {\em collision cross-section} given by
\begin{equation*}
\beta(w, n):=|w\cdot n|\quad \text{for all}\hspace{2mm}w\in\mathbb{R}^{3}, \vspace{2mm}n\in\mathbb{S}^{2}.
\end{equation*}
\end{defn}
We also require the following notion of {\em chaos} for probability distribution functions.
\begin{defn}[Chaos]
We say that a map $F_{0}:\mathcal{D}\rightarrow\mathbb{R}$ is {\bf chaotic} on $\mathcal{D}$ if and only if there exists $\phi_{0}:T\mathbb{R}^{3}\rightarrow\mathbb{R}$ such that
\begin{equation*}
F_{0}(X, V)=\phi_{0}(x, v)\phi_{0}(\ov{x}, \ov{v})\quad \text{for all}\hspace{2mm}[X, V]\in\mathcal{D}.
\end{equation*}
Similarly, we say that a map $F:\mathcal{D}\times (-\infty, \infty)\rightarrow\mathbb{R}$ is chaotic if and only if there exists $\phi:T\mathbb{R}^{3}\times (-\infty, \infty)\rightarrow\mathbb{R}$ such that
\begin{equation*}
F(X, V, t)=\phi(x, v, t)\phi(\ov{x}, \ov{v}, t)
\end{equation*}
for all $[X, V]\in\mathcal{D}$ and $t\in(-\infty, \infty)$.
\end{defn}
We contrast this with the notion of {\em Schwartz chaos}.
\begin{defn}[Schwartz Chaos]\label{schwartzchaos}
We say that $F\in L^{1}_{\mathrm{loc}}(\mathcal{D}\times (-\infty, \infty))$ is {\bf Schwartz chaotic} on $\mathcal{D}\times (-\infty, \infty)$ if and only if there exists a continuous linear operator $W:\mathscr{S}\rightarrow \mathscr{S}'$ such that
\begin{equation*}
F=W'(f\otimes f)\quad \text{in}\hspace{2mm}\mathscr{S}',
\end{equation*}
where $f:T\mathbb{R}^{3}\times (-\infty, \infty)$ is a global-in-time classical solution of the free transport equation 
\begin{equation*}
\partial_{t}f+(p\cdot\nabla_{q})f=0 \quad \text{on}\hspace{2mm}T\mathbb{R}^{3}
\end{equation*}
for $[q, p]\in T\mathbb{R}^{3}$, where $\mathscr{S}$ denotes the class of tempered distributions on $T\mathbb{R}^{6}$.
\end{defn} 
\subsubsection{Statement of Results}
With the above definitions in place, we are ready to state the main results of this work. In section \ref{free}, we prove the following theorem without recourse to the hard sphere flow or to {\em a priori} estimates associated to \eqref{weakdude}.
\begin{thm}[Existence and Uniqueness of Global-in-time Weak Solutions of the Liouville Equation]\label{existandunique}
For any $F_{0}\in C^{1}(\mathcal{D})\cap L^{1}(\mathcal{D})$, there exists a unique physical global-in-time weak solution of \eqref{weakdude}.
\end{thm}
We also have the following structural result which we view as being akin to the {\bf propagation of chaos} in the space of tempered distributions on $T\mathbb{R}^{6}\times (-\infty, \infty)$ for weak solutions of the Liouville equation.
\begin{thm}[Propagation of Schwartz Chaos for $N=2$ Hard Spheres]\label{proppythm}
Suppose an initial datum $F_{0}\in C^{1}(\mathcal{D})\cap L^{1}(\mathcal{D})$ is taken to be of the form
\begin{equation*}
F_{0}(X, V)=\phi_{0}(x, v)\phi_{0}(\ov{x}, \ov{v})\quad \text{for all}\hspace{2mm}[X, V]\in\mathcal{D},
\end{equation*}
for some $\phi_{0}\in C^{1}(T\mathbb{R}^{3})$. The unique physical global-in-time weak solution $F$ of \eqref{weakdude} corresponding to $F_{0}$ is Schwartz chaotic on $\mathcal{D}\times (-\infty, \infty)$.
\end{thm}
Finally, as a straightfoward consequence of theorem \ref{existandunique}, we are able to prove the following result on the global-in-time existence of weak solutions of the BBGKY hierarchy associated to the Liouville equation \eqref{weakdude}.
\begin{cor}[Existence of Global-in-time Weak Solutions of the BBKGY Hierarchy]\label{beebeethm}
Suppose an initial datum $F_{0}\in C^{1}(\mathcal{D})$ is given. There exists an associated global-in-time weak solution of the BBGKY hierarchy which is realised as the average of the unique global-in-time weak solution of the Liouville equation \eqref{weakdude}.
\end{cor}
\subsection{Structure of the Paper}
In section \ref{coney}, we introduce and deduce basic properties of an important structure in phase space, namely {\em collision velocity cones}. We also define the {\em collision time map} and the {\em Boltzmann scattering map} on an important phase sub-bundle. Using these objects, we are then able to write down an explicit form for the hard sphere flow $\{T_{t}\}_{t\in\mathbb{R}}$ on $\mathcal{D}$. In section \ref{derivation}, we use our results of the previous section to derive -- in a systematic manner -- the correct weak formulation of the Liouville equation associated to $H_{\ast}$ above. In section \ref{free}, we introduce our new method of construction of global-in-time weak solutions of \eqref{weakdude}, in the process demonstrating that solutions which begin from chaotic initial data on $\mathcal{D}$ are in fact {\em Schwartz chaotic} in the sense of definition \ref{schwartzchaos} above. In section \ref{bbgkysection}, we derive the BBGKY hierarchy directly from the weak formulation of the Liouville equation by simple judicious choices of test function.
\subsection{Notation}
We employ the square bracket notation $[\cdot, \cdot]:\mathbb{R}^{K}\times\mathbb{R}^{L}\rightarrow\mathbb{R}^{K+L}$ throughout. We write $\Pi_{1}:\mathcal{D}\rightarrow\mathcal{P}$ to denote the spatial projection operator $\Pi_{1}Z:=X$, and $\Pi_{2}:\mathcal{D}\rightarrow$ $\mathbb{R}^{6}$ to denote the velocity projection operator $\Pi_{2}Z:=V$ for all $Z=[X, V]\in\mathcal{D}$. We write $C^{k}(\mathcal{D})$ to denote the set of all bounded $k$-times continuously-differentiable maps on $\mathcal{D}$, while $L^{1}(\mathcal{D})$ denotes the set of Lebesgue-integrable functions ({\em not} equivalence classes of functions). We write $\mathsf{M}(\Omega)$ to denote the set of Borel measures on a subset $\Omega\subseteq\mathbb{R}^{M}$. 
\section{Collision Velocity Cones, Collision Time Maps and Scattering Maps}\label{coney}
The aim of the following section is to follow more traditional lines in the approach to the Liouville equation and write down, in an explicit way, the hard sphere flow $\{T_{t}\}_{t\in\mathbb{R}}$ on the hard sphere phase space $\mathcal{D}$. Let us now define the object we aim to construct. We refer the reader to \textsc{Wilkinson} \cite{wilkinson2018convergence} for the precise definition of physical weak solution of Newton's equations of motion.
\begin{defn}[Hard Sphere Flow on $\mathcal{D}$]\label{hardflow}
The {\bf hard sphere flow} $\{T_{t}\}_{t\in\mathbb{R}}$ on $\mathcal{D}$ is the one-parameter family of operators $T_{t}:\mathcal{D}\rightarrow\mathcal{D}$ defined pointwise on $\mathcal{D}$ in the following manner: for any $Z_{0}\in\mathcal{D}$, the map $t\mapsto Z(t):=T_{t}Z_{0}$ is the unique physical global-in-time weak solution of Newton's equations subject to the initial datum $Z_{0}$.
\end{defn}
One does not expect the dynamics $t\mapsto T_{t}Z_{0}$ to `experience a collision' for every possible choice of $Z_{0}\in\mathcal{D}$. As such, for each $X_{0}\in\mathcal{P}$, one must partition the set of all velocity vectors $\mathcal{V}(X_{0})$ into two classes -- one class of velocities $V_{0}$ for which $t\mapsto T_{t}Z_{0}$ leads to collision, the other class leading only to free space dynamics, i.e.
\begin{equation*}
T_{t}Z_{0}=X_{0}+tV_{0}\quad \text{for all} \hspace{2mm} t\in(-\infty, \infty).
\end{equation*}
Moreover, the characterisation of these classes is important as the {\em collision time map} $\tau$ (as well as the {\em scattering map} $\sigma$) is not defined on the whole phase space bundle $\mathcal{D}$, but only a sub-bundle thereof. This leads to the concept of {\em velocity collision cones}, which we introduce below.
\subsection{Velocity Collision Cones}
For $X\in\mathcal{P}$ and $V\in\mathbb{R}^{6}$, we write $L(X, V)\subset\mathbb{R}^{6}$ to denote the line
\begin{equation*}
L(X, V):=\left\{
X+sV\,:\,s\in\mathbb{R}
\right\}.
\end{equation*}
Moreover, we write $L^{-}(X, V)\subset L(X, V)$ and $L^{+}(X, V)\subset L(X, V)$ to denote the infinite half-lines
\begin{equation*}
\begin{array}{c}
L^{-}(X, V):=\left\{
X+sV\,:\,s\leq 0
\right\}, \vspace{2mm}\\
L^{+}(X, V):=\left\{
X+sV\,:\,s\geq 0
\right\}.
\end{array}
\end{equation*}
For a given $X\in\mathcal{P}$, we define the associated {\bf velocity collision cone} $\mathcal{C}(X)\subset\mathbb{R}^{6}$ as
\begin{equation*}
\mathcal{C}(X):=\left\{
V\in\mathbb{R}^{6}\,:\,L(X, V)\cap\partial\mathcal{P}\neq\varnothing
\right\}.
\end{equation*}
Moreover, the set of all pre-collisional velocities
\begin{equation*}
\mathcal{C}^{-}(X):=\left\{
V\in\mathcal{C}(X)\,:\, L^{+}(X, V)\cap\partial\mathcal{P}\neq\varnothing
\right\},
\end{equation*}
while the set of all post-collisional velocities is given by
\begin{equation*}
\mathcal{C}^{+}(X):=\left\{
V\in\mathcal{C}(X)\,:\, L^{-}(X, V)\cap\partial\mathcal{P}\neq\varnothing
\right\}.
\end{equation*}
%The following proposition records the precise form of the double cone in $\mathbb{R}^{6}$ determined by $\mathcal{C}(X)$, which we state without proof.
%\begin{prop} For any $X\in\mathcal{P}$, one has
%\begin{equation*}
%\mathcal{C}(X)=\left\{
%V\in\mathbb{R}^{6}\,:\,\widetilde{X}\cdot M(X)\leq \frac{\widetilde{X}\cdot V}{|\ov{v}-v|} \leq\frac{1}{\sqrt{2}}
%\right\},
%\end{equation*}
%where $M(X)\in\mathbb{R}^{6}$ is given by
%\begin{equation*}
%M(X):=\frac{1}{\sqrt{1+|\ov{x}-x|^{2}-2\varepsilon}}\left[
%\begin{array}{c}
%\ov{x}-x \\
%\widehat{m}(\ov{x}-x)
%\end{array}
%\right]
%\end{equation*}
%with
%\begin{equation*}
%\widehat{m}(y):=
%\end{equation*}
%and $\widehat{X}\in\mathbb{S}^{5}$ is defined in terms of $X\in\mathcal{P}$ by
%\begin{equation*}
%\widetilde{X}:=\frac{1}{\sqrt{2}|\ov{x}-x|}\left[
%\begin{array}{c}
%-\ov{x}+x\\
%\ov{x}-x
%\end{array}
%\right].
%\end{equation*}
%\end{prop}
It is well known that the maximum number of collisions two hard spheres evolving on the table $\mathcal{P}$ can endure is 2. Indeed, this follows from uniqueness of physical weak solutions of Newton's equations of motion on $\mathcal{D}$. In order to be able to write down the family of hard sphere flow operators $\{T_{t}\}_{t\in\mathbb{R}}$ on $\mathcal{P}$ explicitly, one must first derive the map
\begin{equation*}
\tau:\bigsqcup_{X\in\mathcal{P}}\mathcal{C}(X)\rightarrow\mathbb{R}
\end{equation*}
which yields, for a given initial position $X\in\mathcal{P}$ and velocity $V\in\mathcal{C}(X)$, the unique collision time $\tau(X, V)\in\mathbb{R}$. One must also define the map
\begin{equation*}
\sigma:\bigsqcup_{X\in\mathcal{P}}\mathcal{C}(X)\rightarrow\mathrm{O}(6)
\end{equation*}
which produces the Boltzmann scattering matrix $\sigma(X, V)\in\mathrm{O}(6)$ that maps pre- to post-collisional velocities at collision. Let us set out to derive these maps in the following sections.
\subsection{The Collision Time Map $\tau$}\label{tauderiv}
It is helpful to decompose each fiber $\{X\}\times\mathcal{V}(X)$ over $X\in\mathcal{P}$ of the phase space bundle $\mathcal{D}$ as
\begin{equation*}
\{X\}\times\mathcal{V}(X)=\underbrace{\{X\}\times(\mathbb{R}^{6}\setminus\mathcal{C}(X))}_{\mathcal{F}_{0}(X):=}\cup\underbrace{\{X\}\times\mathcal{C}^{-}(X)}_{\mathcal{F}_{-}(X):=}\cup\underbrace{\{X\}\times\mathcal{C}^{+}(X)}_{\mathcal{F}_{+}(X):=}
\end{equation*}
whenever $X\in\mathcal{P}^{\circ}$. We require no further decomposition for those fibers whose base point $X$ lies on the boundary of the table $\partial\mathcal{P}$. We begin with the following definition. 
\begin{defn}[Collision Time Map]
We write $\tau:\sqcup_{X\in\mathcal{P}}\mathcal{C}(X)\rightarrow\mathbb{R}$ to denote the {\bf collision time map} given by
\begin{equation*}
\tau(X, V):=\mathrm{argmin}\,\left\{
|s|\,:\,X+sV\in\partial\mathcal{P}
\right\}.
\end{equation*}
for $X\in\mathcal{P}$ and $V\in\mathcal{C}(X)$.
\end{defn}
It is possible to write down $\tau$ as an explicit function on the bundle $\sqcup_{X\in\mathcal{P}}\mathcal{C}(X)$. Indeed, we have the following:
\begin{prop}[Characterisation of the Collision Time Map]\label{colltimeprop}
For any $X\in\mathcal{P}$ and $V\in\mathcal{C}(X)$, one has that
\begin{equation*}
\tau(X, V)=\frac{D(X, V)}{|\ov{v}-v|},
\end{equation*}
where
\begin{equation*}
D(X, V):=\left\{
\begin{array}{ll}
-(\ov{x}-x)\cdot(\widehat{\ov{v}-v})-\sqrt{|(\ov{x}-x)\cdot(\widehat{\ov{v}-v})|^{2}+\varepsilon^{2}-|\ov{x}-x|^{2}} & \text{if}\hspace{2mm}V\in\mathcal{C}^{-}(X), \vspace{2mm}\\
-(\ov{x}-x)\cdot(\widehat{\ov{v}-v})+\sqrt{|(\ov{x}-x)\cdot(\widehat{\ov{v}-v})|^{2}+\varepsilon^{2}-|\ov{x}-x|^{2}} & \text{if}\hspace{2mm}V\in\mathcal{C}^{+}(X),
\end{array}
\right.
\end{equation*}
and $\widehat{y}:=y/|y|$ for any $y\in\mathbb{R}^{3}\setminus\{0\}$.
\end{prop}
\begin{proof}
For any $X\in\mathcal{P}$ and $V\in\mathcal{C}(X)$, we herein calculate the length of the portion of the line $L(X, V)$ in $\mathbb{R}^{6}$ which begins at $X$ and ends at the element of $L(X, V)\cap\partial\mathcal{P}$ which is closest in the Euclidean distance to $X$ itself. In this pursuit, it proves useful to define the following auxiliary low-dimensional objects. For any $\varepsilon>0$, we write $\mathbb{R}^{3}_{\varepsilon}\subset\mathbb{R}^{3}$ to denote the set
\begin{equation*}
\mathbb{R}^{3}_{\varepsilon}:=\left\{
y\in\mathbb{R}^{3}\,:\,|y|\geq \varepsilon
\right\}.
\end{equation*}
For any $y\in\mathbb{R}^{3}$ and $w\in\mathbb{R}^{3}$, we define the associated {\em minor line} $\ell(y, w)\subset\mathbb{R}^{3}$ by
\begin{equation*}
\ell(y, w):=\left\{
y+sw\,:\,s\in\mathbb{R}
\right\},
\end{equation*}
together with the minor line segments
\begin{equation*}
\begin{array}{c}
\ell^{-}(y, w):=\left\{
y+sw\,:\,s\leq 0
\right\} \vspace{2mm}\\
\ell^{+}(y, w):=\left\{
y+sw\,:\,s\geq 0
\right\}.
\end{array}
\end{equation*}
In turn, the {\em minor velocity cone} $c(y)\subset\mathbb{R}^{3}$ by
\begin{equation*}
c(y):=\left\{
w\in\mathbb{R}^{3}\,:\,\ell(y, w)\cap\partial\mathbb{R}^{3}_{\varepsilon}\neq \varnothing
\right\},
\end{equation*}
whilst the corresponding pre- and post-collisional minor velocity cones are given by
\begin{equation*}
\begin{array}{c}
c^{-}(y):=\left\{
w\in c(y)\,:\,\ell^{+}(y, w)\cap\mathbb{R}^{3}_{\varepsilon}\neq \varnothing
\right\}, \vspace{2mm}\\
c^{+}(y):=\left\{
w\in c(y)\,:\,\ell^{-}(y, w)\cap\mathbb{R}^{3}_{\varepsilon}\neq \varnothing
\right\},
\end{array}
\end{equation*}
respectively. We record, as a first observation, the following lemma.
\begin{lem}\label{gali}
For any $X=[x, \ov{x}]\in\mathcal{P}$ and $V=[v, \ov{v}]\in\mathcal{C}(X)$, one has that 
$P=[p, \ov{p}]\in L(X, V)\cap\partial\mathcal{P}$ if and only if $\ov{p}-p\in\ell(\ov{x}-x, \ov{v}-v)\cap\mathbb{R}^{3}_{\varepsilon}$.
\end{lem}
\begin{proof}
This follows from a simple Galilean change of coordinates in $\mathbb{R}^{6}$.
\end{proof}
If $y\in\mathbb{R}^{3}_{\varepsilon}$ and $w\in c(y)$ are given, we now consider finding the unique element $p_{\ast}(y, w)$ of $\ell(y, w)\cap\mathbb{R}^{2}_{\varepsilon}$ which is closest to $y\in\mathbb{R}^{2}_{\varepsilon}$ in the Euclidean distance. It is straightforward to show that
\begin{equation}\label{lilpeeone}
p_{\ast}(y, w)=y-\frac{y\cdot w}{|w|^{2}}w-\frac{\sqrt{|y\cdot w|^{2}+(\varepsilon^{2}-|y|^{2})|w|^{2}}}{|w|^{2}}w
\end{equation}
if $w\in c^{-}(y)$, and
\begin{equation}\label{lilpeetwo}
p_{\ast}(y, w)=y-\frac{y\cdot w}{|w|^{2}}w+\frac{\sqrt{|y\cdot w|^{2}+(\varepsilon^{2}-|y|^{2})|w|^{2}}}{|w|^{2}}w
\end{equation}
if $w\in c^{+}(y)$.
We write $P_{\ast}(X, V)$ to denote the point of the hard sphere table $\mathcal{P}$ determined from the point $p_{\ast}(\ov{x}-x, \ov{v}-v)$ given in lemma \ref{gali} above. From the above expressions \eqref{lilpeeone} and \eqref{lilpeetwo}, we quickly deduce that the {\em signed distance} $D(X, V)\in\mathbb{R}$ from $X$ to $P_{\ast}(X, V)$ is simply given by
\begin{equation*}
D(X, V)=\left\{
\begin{array}{ll}
-(\ov{x}-x)\cdot(\widehat{\ov{v}-v})-\sqrt{|(\ov{x}-x)\cdot(\widehat{\ov{v}-v})|^{2}+\varepsilon^{2}-|\ov{x}-x|^{2}} & \text{if}\hspace{2mm}V\in\mathcal{C}^{-}(X), \vspace{2mm}\\
-(\ov{x}-x)\cdot(\widehat{\ov{v}-v})+\sqrt{|(\ov{x}-x)\cdot(\widehat{\ov{v}-v})|^{2}+\varepsilon^{2}-|\ov{x}-x|^{2}} & \text{if}\hspace{2mm}V\in\mathcal{C}^{+}(X),
\end{array}
\right.
\end{equation*}
whence the collision time $\tau(X, V)\in\mathbb{R}$ is given by
\begin{equation*}
\tau(X, V):=\frac{D(X, V)}{|\ov{v}-v|}
\end{equation*}
for all $X\in\mathcal{P}$ and $V\in\mathbb{R}^{6}$, as claimed in the statement of the proposition. 
\end{proof}
We conclude this small section on the collision time map with the following important identity, which shall be of repeated use in the sequel.
\begin{prop}[Transport Identity I]
For any $X\in\mathcal{P}^{\circ}$ and $V\in\mathcal{C}(X)$, one has that
\begin{equation}\label{transportidentity}
(V\cdot \nabla_{X})\tau(X, V)=-1.\tag{ID--I}
\end{equation}
\end{prop}
\begin{proof}
This is a simple calculation which is left to the reader. 
\end{proof}
With this work in place, we are now able to proceed to the definition and properties of the Boltzmann scattering map $\sigma:\sqcup_{X\in\mathcal{P}}\mathcal{C}(X)\rightarrow\mathrm{O}(6)$.
\subsection{The Boltzmann Scattering Map $\sigma$}
In the case of $N=2$ hard spheres, the boundary of the hard sphere table $\partial\mathcal{P}$ is determined as the level set of a real analytic function on $\mathbb{R}^{6}$, namely
\begin{equation*}
\partial\mathcal{P}=\left\{
X\in\mathbb{R}^{6}\,:\,F^{\varepsilon}(X)=0
\right\},
\end{equation*}
where $F^{\varepsilon}:\mathbb{R}^{6}\rightarrow\mathbb{R}$ is defined pointwise as $F^{\varepsilon}(X):=|x-\ov{x}|^{2}-\varepsilon^{2}$ for $X=[x, \ov{x}]\in\mathcal{P}$. We now employ this characterisation of the boundary of the hard sphere table to write down an important unit-norm field on $\sqcup_{X\in\mathcal{P}}\mathcal{C}(X)$. 
\begin{defn}[Outward Normal Map $\widehat{\nu}$]
For $X\in\mathcal{P}$ and $V\in\mathcal{C}(X)$, we write $\widehat{\nu}(X, V)\in\mathbb{S}^{5}$ to denote the unit-norm vector which is normal to the boundary point $X+\tau(X, V)V\in\partial\mathcal{P}$ of the hard sphere table, i.e.
\begin{equation*}
\widehat{\nu}(X, V):=\frac{\nabla F^{\varepsilon}(X+\tau(X, V)V)}{|\nabla F^{\varepsilon}(X+\tau(X, V)V)|}.
\end{equation*}
\end{defn}
By employing techniques analogous to those used in section \eqref{tauderiv} above, one may verify quickly that the map $\widehat{\nu}:\sqcup_{X\in\mathcal{P}}\mathcal{C}(X)\rightarrow\mathbb{S}^{5}$ is given pointwise by
\begin{equation*}
\widehat{\nu}(X, V)=\frac{1}{\sqrt{2}\varepsilon}\left[
\begin{array}{c}
\ov{x}-x+\tau(X, V)(\ov{v}-v) \\
x-\ov{x}+\tau(X, V)(v-\ov{v})
\end{array}
\right]
\end{equation*}
for $X\in\mathcal{P}$ and $V\in\mathcal{C}(X)$. In turn, we define the {\bf Boltzmann scattering map} $\sigma:\sqcup_{X\in\mathcal{P}}\mathcal{C}(X)\rightarrow\mathrm{O}(6)$ pointwise by
\begin{equation*}
\sigma(X, V):=I-2\widehat{\nu}(X, V)\otimes \widehat{\nu}(X, V).
\end{equation*}
Intuitively, for a given `initial point' $X\in\mathcal{P}$ with $V\in\mathcal{C}(X)$, the map $\sigma: (X, V)\mapsto \sigma(X, V)$ returns the scattering matrix in $\mathrm{O}(6)$ which becomes `active' at the unique time when the hard sphere trajectory $t\mapsto X+tV$ `hits' the boundary of the table. Of course, this means that $\sigma$ is not defined as a map into $\mathrm{O}(6)$ for all $[X, V]\in \mathcal{D}$. As it will be of importance later, we perform a natural extension of this map in section \ref{extension} below. In any case, one may verify that the conservation of linear momentum
\begin{equation*}
\left(
\begin{array}{c}
(\sigma(X, V)V)_{1}+(\sigma(X, V)V)_{4}\\
(\sigma(X, V)V)_{2}+(\sigma(X, V)V)_{5} \\
(\sigma(X, V)V)_{3}+(\sigma(X, V)V)_{6}
\end{array}
\right)=\left(
\begin{array}{c}
V_{1}+V_{4} \\
V_{2}+ V_{5} \\
V_{3}+ V_{6}
\end{array}
\right), \tag{COLM}
\end{equation*}
the conservation of angular momentum (for all points of measurement $a\in\mathbb{R}^{3}$) given by
\begin{align*}
\left(
\begin{array}{c}
X_{1}-a_{1} \\
X_{2}-a_{2} \\
X_{3}-a_{3}
\end{array}
\right)\wedge\left(
\begin{array}{c}
(\sigma(X, V)V)_{1} \\
(\sigma(X, V)V)_{2} \\
(\sigma(X, V)V)_{3}
\end{array}
\right)+\left(
\begin{array}{c}
X_{4}-a_{1} \\
X_{5}-a_{2} \\
X_{6}-a_{3}
\end{array}
\right)\wedge\left(
\begin{array}{c}
(\sigma(X, V)V)_{4} \\
(\sigma(X, V)V)_{5} \\
(\sigma(X, V)V)_{6}
\end{array}
\right) \vspace{2mm}\\
=\left(
\begin{array}{c}
X_{1}-a_{1} \\
X_{2}-a_{2} \\
X_{3}-a_{3}
\end{array}
\right)\wedge\left(
\begin{array}{c}
V_{1} \\
V_{2} \\
V_{3}
\end{array}
\right)+\left(
\begin{array}{c}
X_{4}-a_{1} \\
X_{5}-a_{2} \\
X_{6}-a_{3}
\end{array}
\right)\wedge\left(
\begin{array}{c}
V_{4} \\
V_{5} \\
V_{6}
\end{array}
\right) \tag{COAM} 
\end{align*}
and the conservation of kinetic energy
\begin{equation*}
\frac{1}{2}|\sigma(X, V)V|^{2}=\frac{1}{2}|V|^{2} \tag{COKE}
\end{equation*}
hold true for all $X\in\mathcal{P}$ and $V\in\mathcal{C}(X)$. 
\subsection{Extension of the Boltzmann Scattering Map}\label{extension}
One obvious barrier to the extension of $\sigma$ from $\sqcup_{X\in\mathcal{P}}\mathcal{C}(X)$ to $\mathcal{D}$ is that the collision time map $\tau$ is not defined on those points $Z=[X, V]\in\mathcal{D}$ for which $t\mapsto T_{t}Z$ admits no collisions on $(-\infty, \infty)$. Owing to the fact that $\sigma$ acts simply as the identity matrix on the boundary of velocity cones $\mathcal{C}(X)$, we establish the following definition.
\begin{defn}[Extension of Boltzmann Scattering Map]
We write $\sigma_{\ast}:\mathcal{D}\rightarrow\mathrm{O}(6)$ to denote the map on $\mathcal{D}$ defined pointwise by
\begin{equation}
\sigma_{\ast}(X, V):=\left\{
\begin{array}{ll}
\sigma(X, V) & \quad \text{if}\hspace{2mm}[X, V]\in\sqcup_{X\in\mathcal{P}}\mathcal{C}(X), \vspace{2mm}\\
I & \quad \text{otherwise},
\end{array}
\right.
\end{equation}
where $I\in\mathbb{R}^{6\times 6}$ denotes the identity matrix.
\end{defn}
Using this extension, we in turn define the map $\Sigma_{\ast}:\mathcal{D}\rightarrow\mathcal{D}$ pointwise by
\begin{equation*}
\Sigma_{\ast}(Z):=\left(
\begin{array}{cc}
\sigma_{\ast}(X, V) & 0 \\
0 & \sigma_{\ast}(X, V)
\end{array}
\right)Z,
\end{equation*}
for $Z=[X, V]\in\mathcal{D}$, where $0\in\mathbb{R}^{6\times 6}$ is the zero matrix. We also record the following useful lemma on a basic property of the collision time map $\tau$ on $\sqcup_{X\in\mathcal{P}}\mathcal{C}(X)$.
\begin{lem}[Symmetry of Collision Time Map]\label{timesymm}
For any $X\in\mathcal{P}$ and $V\in\mathcal{C}(X)$, one has that
\begin{equation*}
\tau(X', V')=\tau(X, V),
\end{equation*}
where $[X', V']\in\mathcal{D}$ denotes the element
\begin{equation*}
\left[
\begin{array}{c}
X' \\
V'
\end{array}
\right]:=\Sigma_{\ast}(X, V)\left[
\begin{array}{c}
X \\
V
\end{array}
\right].
\end{equation*}
\end{lem}
The map $\sigma_{\ast}$ will be used in section \ref{free} to build a coordinate transformation of phase space $\mathcal{D}$. As an immediate consequence of lemma \ref{timesymm}, one can quickly verify the following result.
\begin{prop}
The map $\Sigma_{\ast}:\mathcal{D}\rightarrow\mathcal{D}$ is a smooth involution. Moreover, 
\begin{equation*}
|\mathrm{det}\,D\Sigma_{\ast}(Z)|=1
\end{equation*}
for all $X\in\mathcal{P}^{\circ}$ and $V\in\mathbb{R}^{6}$.
\end{prop}
\begin{proof}
This follows from a simple (but long and laborious) calculation, which we leave to the reader.
\end{proof}
Let us finally record the following important identity, whose demonstration we leave once again to the reader. It shall be used repeatedly in the sequel in the derivation of the weak formulation of the Liouville equation \eqref{weakdude}.
\begin{prop}[Transport Identity II]
For any $X\in\mathcal{P}^{\circ}$ and $V\in\mathbb{R}^{6}$, it holds that
\begin{equation*}
(V\cdot\nabla_{X})\sigma_{\ast}(X, V)=0\quad \text{in}\hspace{2mm}\mathbb{R}^{6\times 6}. \tag{ID--II}
\end{equation*}
\end{prop}
We are now in a position to construct the hard sphere flow on $\mathcal{D}$ (as per definition \ref{hardflow} above) `by hand'.
\subsection{Definition of the Hard Sphere Flow Maps $T_{t}$}\label{flowsection}
When defining the family of flow operators $\{T_{t}\}_{t\in\mathbb{R}}$ on $\mathcal{D}$ systematically, it is helpful to make use of the partition of the fibers of the phase space bundle
\begin{equation*}
\mathcal{P}=\bigsqcup_{X\in\mathcal{P}}\mathcal{V}(X)
\end{equation*}
that was outlined at the beginning of section \ref{tauderiv} above. Moreover, as the phase space bundle has fibers of the shape
\begin{equation*}
\{X\}\times \Sigma_{X}
\end{equation*}
for half-spaces $\Sigma_{X}\subset\mathbb{R}^{6}$ when $X\in\partial\mathcal{P}$, yet fibers of the shape
\begin{equation*}
\{X\}\times \mathbb{R}^{6}
\end{equation*}
when $X\in\mathcal{P}^{\circ}$, it is prudent to split our considerations into the following cases:
\subsubsection{Collision-free Dynamics}
Suppose $X\in\mathcal{P}^{\circ}$ is given. Firstly, we define the flow operator $T_{t}:(X, V)\mapsto\mathcal{D}$ factorwise by
\begin{equation*}
(\Pi_{1}\circ T_{t})Z:=X+tV
\end{equation*}
and
\begin{equation*}
(\Pi_{2}\circ T_{t})Z:=V
\end{equation*}
if $V\in\mathbb{R}^{6}$ is such that $(X, V)\in\mathcal{F}_{0}(X)$. 
\subsubsection{Collisional Dynamics: $X\in\mathcal{P}^{\circ}$}
Secondly, we set
\begin{equation*}
(\Pi_{1}\circ T_{t})Z:=\left\{
\begin{array}{ll}
X+tV & \quad \text{if}\hspace{2mm}-\infty<t\leq \tau(X, V), \vspace{2mm}\\
X+\tau(X, V)V+(t-\tau(X, V))\sigma(X, V)V & \quad \text{if}\hspace{2mm}\tau(X, V)<t<\infty,
\end{array}
\right.
\end{equation*}
and
\begin{equation*}
(\Pi_{2}\circ T_{t})Z:=\left\{
\begin{array}{ll}
V & \quad \text{if}\hspace{2mm}-\infty<t\leq\tau(X, V), \vspace{2mm}\\
\sigma(X, V)V & \quad \text{if}\hspace{2mm}\tau(X, V)<t<\infty,
\end{array}
\right.
\end{equation*}
if $V\in\mathbb{R}^{6}$ is taken with $(X, V)\in\mathcal{F}_{-}(X)$. Thirdly, we define it to be
\begin{equation*}
(\Pi_{1}\circ T_{t})Z:=\left\{
\begin{array}{ll}
X+\tau(X, V)V+(t-\tau(X, V))\sigma(X, V)V & \quad \text{if}\hspace{2mm}-\infty<t\leq \tau(X, V), \vspace{2mm}\\
X+tV & \quad \text{if}\hspace{2mm}\tau(X, V)<t<\infty,
\end{array}
\right.
\end{equation*}
and
\begin{equation*}
(\Pi_{2}\circ T_{t})Z:=\left\{
\begin{array}{ll}
\sigma(X, V)V & \quad \text{if}\hspace{2mm}-\infty<t\leq\tau(X, V), \vspace{2mm}\\
V & \quad \text{if}\hspace{2mm}\tau(X, V)<t<\infty,
\end{array}
\right.
\end{equation*}
if $V\in\mathbb{R}^{6}$ is such that $(X, V)\in\mathcal{F}_{+}(X)$. 
\subsubsection{Collisional Dynamics: $X\in\partial\mathcal{P}$}
Finally, if $X\in\partial\mathcal{P}$, we set
\begin{equation*}
(\Pi_{1}\circ T_{t})Z:=\left\{
\begin{array}{ll}
X+tV & \quad \text{if}\hspace{2mm}-\infty<t\leq 0, \vspace{2mm}\\
X+t\sigma(X, V)V & \quad \text{if}\hspace{2mm}0<t<\infty,
\end{array}
\right.
\end{equation*}
and
\begin{equation*}
(\Pi_{2}\circ T_{t})Z:=\left\{
\begin{array}{ll}
V & \quad \text{if}\hspace{2mm}-\infty<t\leq 0, \vspace{2mm}\\
\sigma(X, V)V & \quad \text{if}\hspace{2mm}0<t<\infty.
\end{array}
\right.
\end{equation*}
With the above definition of the flow $\{T_{t}\}_{t\in\mathbb{R}}$ in place we shall now gather together, for the convenience of the reader, some of its elementary properties.
\begin{prop}
Suppose $Z_{0}\in\mathcal{D}$. The maps $X\in C^{0}(\mathbb{R}, \mathcal{P})$ defined pointwise by $X(t):=(\Pi_{1}\circ T_{t})Z_{0}$ and $V\in\mathrm{BV}_{\mathrm{loc}}(\mathbb{R}, \mathbb{R}^{6})$ given by $V(t):=(\Pi_{2}\circ T_{t})Z_{0}$ for each $t\in\mathbb{R}$ constitute a physical global-in-time weak solution of Newton's equations.
\end{prop}
Next, we state the fact that the hard sphere flow is one-parameter family of bijective maps which preserve the measure $\mathscr{L}_{12}\mres\mathcal{D}$.
\begin{prop}\label{measureprop}
For each $t\in(-\infty, \infty)$ the map $T_{t}:\mathcal{D}\rightarrow\mathcal{D}$ is a bijection which is measure preserving in the sense that
\begin{equation*}
T_{t}\#(\mathscr{L}_{12}\mres\mathcal{D})=\mathscr{L}_{12}\mres\mathcal{D}\quad \text{in}\hspace{2mm}\mathsf{M}(\mathcal{D}),
\end{equation*}
where $\#$ denotes the pushforward operator.
\end{prop}
\begin{proof}
See \textsc{Cercignani, Illner and Pulvirenti} (\cite{cercignani2013mathematical}, chapter 2).
\end{proof}
With the above properties in place, we define the map $F:\mathcal{D}\times(-\infty, \infty)\rightarrow\mathbb{R}$ by
\begin{equation}\label{bigeff}
F(X, V, t):=F_{0}((\Pi_{1}\circ T_{-t})Z), (\Pi_{2}\circ T_{-t})Z).
\end{equation}
for $X\in\mathcal{P}$ and $V\in\mathbb{R}^{6}$, where $\Pi_{1}:Z\mapsto X$ and $\Pi_{2}:Z\mapsto V$ are projection operators. We are interested in the sequel in demonstrating that the mild solution \eqref{bigeff} constitutes a weak solution of the Liouville equation.
\section{Derivation of the Liouville Equation}\label{derivation}
In \textsc{Gallagher, Saint-Raymond and Texier} (\cite{gallagher2014newton}, sections 4.2 and 4.3), the authors perform a derivation of the hierarchy (and thereby the Liouville equation) by assuming that both the distribution function $F$ and also each test function $\Phi$ adhere to the `boundary conditions' that
\begin{equation}\label{wrongbc}
\begin{array}{c}
F(Z^{\mathrm{in}}, t)=F(Z^{\mathrm{out}}, t), \\
\Phi(Z^{\mathrm{in}}, t)=\Phi(Z^{\mathrm{out}}, t),
\end{array}
\end{equation}
for all $t\in(-\infty, \infty)$, where, in the notation of this article, one has that
\begin{equation}\label{constraintlsr}
\begin{array}{c}
Z^{\mathrm{in}}=[X^{\mathrm{in}}, V^{\mathrm{in}}]\in\partial\mathcal{D}\\
Z^{\mathrm{out}}:=[X^{\mathrm{in}}, \sigma(X^{\mathrm{in}}, V^{\mathrm{in}})V]\notin\partial\mathcal{D}.
\end{array}
\end{equation}
We view this `boundary condition' rather as a {\em constraint} on mappings. We prefer, in this work, to reformulate the approach in \cite{gallagher2014newton} and understand the behaviour of the solution $F$ of the Liouville equation at the boundary $\partial\mathcal{D}$ in terms of discontinuous probability {\bf mass transport}. 

Our approach is readily exemplified by the problem of furnishing the Liouville equation with singular measure-valued initial data of the type
\begin{equation}\label{deerak}
\mu_{0}:=\delta_{Z_{0}}\quad \text{in}\hspace{2mm}\mathsf{M}(\mathcal{D})
\end{equation}
for some $Z_{0}\in\mathcal{D}$ with $V_{0}\in \mathcal{C}(X_{0})$. Intuitively (and without defining the precise sense in which the Liouville equation is satisfied by measure-valued data), the unique physical solution of the Liouville equation associated to \eqref{deerak} is given by the one-parameter family of measures $\{\mu_{t}\}_{t\in\mathbb{R}}$ with 
\begin{equation*}
\mu_{t}:=\delta_{T_{t}Z_{0}}\quad \text{for each}\hspace{2mm}t\in(-\infty, \infty).
\end{equation*}
Given that the flow is such that $t\mapsto T_{t}Z_{0}$ is a lower semi-continuous map on $(-\infty, \infty)$, it is evident that the mass of the measure $\mu_{\tau(Z_{0})}$ is transported discontinuously to $T_{t}Z_{0}$ for $t>\tau(Z_{0})$, where
\begin{equation*}
\lim_{t\rightarrow\tau(Z_{0})}\left|(\Pi_{2}\circ T_{Z_{0}})Z_{0}-(\Pi_{2}\circ T_{t})Z_{0}\right|>0.
\end{equation*}
In particular, the probability mass of the solution is not, as perhaps suggested by \eqref{wrongbc} above, shared by two points in phase space simultaneously when $t=\tau(Z_{0})$. 

Extending this picture to those initial data $\mu_{0}$ which are continuous probability distribution functions on $\mathcal{D}$ (as opposed to singular measures thereon), we find ourselves led to the problem of understanding what PDE the measure $\mu\in\mathsf{M}(\mathcal{D}\times (-\infty, \infty))$ characterised -- via the Riesz-Kakutani representation theorem -- by the duality relation
\begin{equation*}
\int_{\mathcal{D}\times (-\infty, \infty)}\Phi\,d\mu=\int_{\mathcal{D}}\left(\int_{-\infty}^{\infty}\Phi(Z, t)\,d\mu_{t}(Z)\right)\,dt \quad \text{for all}\hspace{2mm}\Phi\in C^{0}(\mathcal{D}\times(-\infty, \infty)),
\end{equation*}
with $\mu_{t}:=F(\cdot, t)\mathscr{L}_{12}\mres\mathcal{D}$, satisfies in the sense of distributions on $\mathcal{D}\times (-\infty, \infty)$. The fact that probability mass of the initial distribution function $F_{0}$ is transported discontinuously across the phase bundle $\mathcal{D}$ will compel us to decompose phase space in an amenable manner.
\begin{rem}
To wit, in this work, we do not place the constraint \eqref{constraintlsr} on any space of test functions. Moreover, we also need not enforce that the test function is a {\em symmetric} function of its variables.
\end{rem}
\subsection{Heuristics}
Suppose an initial probability distribution function $F_{0}\in C^{1}(\mathcal{D})\cap L^{1}(\mathcal{D})$ is given. We would now like to derive the evolution equation which the associated absolutely continuous probability measure
\begin{equation*}
(F\mathscr{L}_{12}\mres\mathcal{D})\otimes \mathscr{L}_{1}
\end{equation*}
satisfies in the sense of distributions on $\mathcal{D}\times(-\infty, \infty)$, where $F$ is the map defined in \eqref{bigeff} above. To begin in this direction, for any $\Phi\in C^{\infty}_{c}(\mathcal{D}\times(-\infty, \infty))$, it is natural to consider the evaluation of the integral expressions
\begin{equation}\label{iint}
I(\Phi)=\int_{\mathcal{P}}\int_{\mathbb{R}^{6}}\int_{-\infty}^{\infty}F(X, V, t)\partial_{t}\Phi(X, V, t)\,dtdVdX
\end{equation}
and
\begin{equation}\label{jint}
J(\Phi)=\int_{\mathcal{P}}\int_{\mathbb{R}^{6}}\int_{-\infty}^{\infty}F(X, V, t)(V\cdot\nabla_{X})\Phi(X, V, t)\,dtdVdX
\end{equation}
in anticipation of deriving a PDE of the type
\begin{equation*}
\partial_{t}F+(V\cdot\nabla_{X})F=\text{term which captures collisions},
\end{equation*}
whose sense is to be understood in that of distributions on $\mathcal{D}\times(-\infty, \infty)$. We shall show that one has that
\begin{equation}\label{resultant}
I(\Phi)+J(\Phi)=-\int_{\partial\mathcal{P}}\int_{\mathbb{R}^{6}}\int_{-\infty}^{\infty}F(Y, V, t)\Phi(Y, V, t)V\cdot\widehat{\nu}_{\ast}(Y)\,dtdVd\mathscr{H}(Y),
\end{equation}
where $\widehat{\nu}_{\ast}:\partial\mathcal{P}\rightarrow\mathbb{S}^{5}$ denotes the outward normal map to the boundary of the hard sphere table $\mathcal{P}$ given by
\begin{equation*}
\widehat{\nu}_{\ast}(Y)=\frac{1}{\sqrt{2}}\left[
\begin{array}{c}
-\widehat{n}(Y) \\
\widehat{n}(Y)
\end{array}
\right],
\end{equation*}
with $\widehat{n}:\partial\mathcal{P}\rightarrow\mathbb{S}^{2}$ defined pointwise by
\begin{equation*}
\widehat{n}(Y):=\frac{\ov{y}-y}{|\ov{y}-y|}\quad \text{for}\hspace{2mm}Y=[y, \ov{y}]\in\partial\mathcal{P}.
\end{equation*}
The reader will note that the integral in \eqref{resultant} involves the trace of the distribution function $F$ on subset of $\mathcal{D}$ of $\mathscr{L}_{12}\mres\mathcal{D}$-measure zero. In particular, the resultant \eqref{resultant} would admit no immediate sense were we to instead endow the system with initial data in Lebesgue spaces or Sobolev spaces, for instance. It is possible to assign a meaning to those solutions of the Liouville equation whose initial data are much less regular than simply $C^{0}(\mathcal{D})$. We choose not to pursue this here, as it serves only to obscure our geometric argument below. Let us now advance to processing the information contained in the integral $I(\Phi)$ above.
\subsection{The Time Derivative Term}
Let $F_{0}\in C^{1}(\mathcal{D})\cap L^{1}(\mathcal{D})$ be given. In general, the pushforward map $F$ is not smooth on $\mathcal{D}\times (-\infty, \infty)$. To begin, we consider the time derivative contribution to the LHS of the formal Liouville equation given by
\begin{equation*}
I(\Phi)=\int_{\mathcal{P}}\int_{\mathbb{R}^{6}}\int_{-\infty}^{\infty}F(X, V, t)\partial_{t}\Phi(X, V, t)dtdVdX,
\end{equation*}
where $\Phi\in C^{\infty}_{c}(T\mathbb{R}^{6}\times (-\infty, \infty))$. The aim herein is to write the integral over those regions of $\mathcal{D}\times(\infty, \infty)$ on which the map $F$ is a smooth function, so that one may transfer derivatives from the test function $\Phi$ to the initial datum $F_{0}$. Indeed, guided by the definition of the flow map $T_{t}:\mathcal{D}\rightarrow\mathcal{D}$ in section \ref{flowsection} above, one may decompose the integral $I(\Phi)$ as
\begin{equation*}
\begin{array}{c}
I(\Phi)=\displaystyle \underbrace{\int_{\mathcal{P}}\int_{\mathbb{R}^{6}\setminus\mathcal{C}(X)}\int_{-\infty}^{\infty}F(X, V, t)\partial_{t}\Phi(X, V, t)\,dtdVdX}_{I_{1}(\Phi):=} \vspace{2mm} \\
+\displaystyle \underbrace{\int_{\mathcal{P}}\int_{\mathcal{C}(X)}\int_{-\infty}^{\infty}F(X, V, t)\partial_{t}\Phi(X, V, t)\,dtdVdX}_{I_{2}(\Phi):=}.
\end{array}
\end{equation*}
Owing to the absence of collisions for the dynamics $t\mapsto \Pi_{t}Z$ whenever $Z$ is such that $V\in\mathbb{R}^{6}\setminus\mathcal{C}(X)$, one finds that
\begin{equation*}
\begin{array}{c}
\displaystyle I_{1}(\Phi)=\int_{\mathcal{P}}\int_{\mathbb{R}^{6}\setminus\mathcal{C}(X)}\int_{-\infty}^{\infty}F_{0}(\Pi_{-t}Z)\partial_{t}\Phi(X, V, t)\,dtdVdX \vspace{2mm}\\
\displaystyle =\int_{\mathcal{P}}\int_{\mathbb{R}^{6}\setminus\mathcal{C}(X)}\int_{-\infty}^{\infty}F_{0}(X-tV, V)\partial_{t}\Phi(X, V, t)\,dtdVdX,
\end{array}
\end{equation*}
from which one deduces
\begin{equation*}
I_{1}(\Phi)=\int_{\mathcal{P}}\int_{\mathbb{R}^{6}\setminus\mathcal{C}(X)}\int_{-\infty}^{\infty}V\cdot\nabla_{X} F_{0}(X-tV, V)\Phi(X, V, t)\,dtdVdX.
\end{equation*}
The remaining contribution $I_{2}(\Phi)$ to the integral $I(\Phi)$ cannot be evaluated in such a straightforward manner, due to the fact that collisions transport the mass of $F_{0}$ across phase space $\mathcal{D}$ in a discontinuous manner. 

We split those integrals over subsets of $\mathcal{C}(X)$ using the fiberwise decomposition of section \ref{tauderiv} above. By employing the simple observation that each velocity-time cone $\mathcal{C}(X)\times\mathbb{R}$ admits the decomposition
\begin{equation*}
\mathcal{C}(X)\times\mathbb{R}=\mathcal{R}^{(-, -)}(X)\cup\mathcal{R}^{(-, +)}(X)\cup\mathcal{R}^{(+, -)}(X)\cup\mathcal{R}^{(+, +)}(X),
\end{equation*}
where
\begin{equation*}
\begin{array}{c}
\mathcal{R}^{(-, -)}(X):=\left\{
(V, t)\in\mathcal{C}^{-}(X)\times(-\infty, \infty)\,:\,-t\leq\tau(X, V)
\right\}, \vspace{2mm}\\
\mathcal{R}^{(-, +)}(X):=\left\{
(V, t)\in\mathcal{C}^{-}(X)\times(-\infty, \infty)\,:\,-t>\tau(X, V)
\right\}, \vspace{2mm}\\
\mathcal{R}^{(+, -)}(X):=\left\{
(V, t)\in\mathcal{C}^{+}(X)\times(-\infty, \infty)\,:\,-t\leq\tau(X, V)
\right\}, \vspace{2mm} \\
\mathcal{R}^{(+, +)}(X):=\left\{
(V, t)\in\mathcal{C}^{+}(X)\times(-\infty, \infty)\,:\,-t>\tau(X, V)
\right\},
\end{array}
\end{equation*}
it follows that 
\begin{equation*}
\begin{array}{c}
\displaystyle I_{2}(\Phi)=\underbrace{\int_{\mathcal{P}}\int_{\mathcal{R}^{(-, -)}(X)}F(X, V, t)\partial_{t}\Phi(X, V, t)\,d\mu_{X}^{(-, -)}dX}_{I_{2}^{(-, -)}(\Phi):=} \vspace{2mm} \\
+\displaystyle \underbrace{\int_{\mathcal{P}}\int_{\mathcal{R}^{(-, +)}(X)}F(X, V, t)\partial_{t}\Phi(X, V, t)\,d\mu_{X}^{(-, +)}dX}_{I_{2}^{(-, +)}(\Phi):=} + \underbrace{\int_{\mathcal{P}}\int_{\mathcal{R}^{(+, -)}(X)}F(X, V, t)\partial_{t}\Phi(X, V, t)\,d\mu_{X}^{(+, 1)}dX}_{I_{2}^{(+, 1)}(\Phi):=} \vspace{2mm}\\
\displaystyle +\underbrace{\int_{\mathcal{P}}\int_{\mathcal{R}^{(+, +)}(X)}F(X, V, t)\partial_{t}\Phi(X, V, t)\,d\mu_{X}^{(+, +)}dX}_{I_{2}^{(+, +)}(\Phi):=}
\end{array}
\end{equation*}
where $\mu_{X}^{(\Box, \ocircle)}$ is the $X$-parametrised Borel measure on $\mathcal{R}^{(\Box, \ocircle)}(X)$ given by
\begin{equation*}
\mu_{X}^{(\Box, \ocircle)}:=\mathscr{L}_{6}\otimes\mathscr{L}_{1}\mres\mathcal{R}_{X}^{(\Box, \ocircle)} \quad \text{in}\hspace{2mm}\mathsf{M}(\mathbb{R}^{6}\times (-\infty, \infty))
\end{equation*}
for $\Box, \ocircle\in\{-, +\}$. In order to `pass the time derivative' from the test function $\Phi$ to $F$ by way of integration by parts, it is at this point one must appeal carefully to (i) the definition of the flow map $t\mapsto T_{t}$, and (ii) the collision map $\tau:\sqcup_{X\in\mathcal{P}}\mathcal{C}(X)\rightarrow\mathbb{R}$. To fix ideas, we consider the particular case of the integral $I^{(-, -)}_{2}(\Phi)$. In this case, one has
\begin{equation*}
\begin{array}{l}
\displaystyle I_{2}^{(-, -)}(\Phi)=\int_{\mathcal{P}}\int_{\mathcal{R}^{(-, -)}(X)}F_{0}(\Pi_{-t}Z)\partial_{t}\Phi(Z, t)\,d\mu_{X}^{-}(V, t)dX \vspace{2mm}\\
=\displaystyle\int_{\mathcal{P}}\int_{\mathcal{C}^{-}(X)}\int_{-\infty}^{-\tau(Z)}F_{0}(X-tV, V)\partial_{t}\Phi(X, V, t)\,dtdVdX \vspace{2mm} \\
=\displaystyle\int_{\mathcal{P}}\int_{C^{-}(X)}F_{0}(X+\tau(Z)V, V)\Phi(X, V, -\tau(Z))\,dVdX \vspace{2mm}\\
+\displaystyle\int_{\mathcal{P}}\int_{\mathcal{R}^{(-, -)}(X)}V\cdot\nabla_{X}F_{0}(X-tV, V)\Phi(X, V, t)\,d\mu_{X}^{(-, -)}dX.
\end{array}
\end{equation*}
Similarly, in the case of the integral $I_{2}^{(-, +)}(\Phi)$, one has that
\begin{equation*}
\begin{array}{l}
\displaystyle I_{2}^{(+, -)}(\Phi)=\int_{\mathcal{P}}\int_{\mathcal{R}^{(-, +)}(X)}F_{0}(\Pi_{-t}Z)\partial_{t}\Phi(Z, t)\,d\mu_{X}^{(-, +)}(V, t)dX \vspace{2mm}\\
=\displaystyle\int_{\mathcal{P}}\int_{\mathcal{C}^{-}(X)}\int_{-\tau(Z)}^{\infty}F_{0}(X+\tau(Z)V-(t+\tau(Z))\sigma(Z)V, \sigma(Z)V)\partial_{t}\Phi(X, V, t)\,dtdVdX \vspace{2mm} \\
=-\displaystyle\int_{\mathcal{P}}\int_{C^{+}(X)}F_{0}(X+\tau(Z)V, \sigma(Z)V)\Phi(X, V, -\tau(Z))\,dVdX \vspace{2mm}\\
+\displaystyle\int_{\mathcal{P}}\int_{\mathcal{R}^{(-, +)}(X)}\sigma(Z)V\cdot\nabla_{X} F_{0}(X+\tau(Z)V-(t+\tau(Z))\sigma(Z)V, \sigma(Z)V)\Phi(X, V, t)\,d\mu_{X}^{+}dX.
\end{array}
\end{equation*}
More generally, one has the result of the following proposition, whose proof may be obtained in a straightforward manner.
\begin{prop}\label{temporalid}
Suppose $F_{0}\in C^{1}(\mathcal{D})\cap L^{1}(\mathcal{D})$. For any test function $\Phi\in C^{1}_{c}(T\mathbb{R}^{6}\times (-\infty, \infty))$, the map $F$ defined by \eqref{bigeff} satisfies the following identities:
\begin{equation}\label{teeoneminus}
\begin{array}{c}
\displaystyle \int_{\mathcal{P}}\int_{\mathcal{C}^{-}(X)}\int_{-\infty}^{-\tau(Z)}F(X, V, t)\partial_{t}\Phi(X, V, t)\,dtdVdX \vspace{2mm}\\
=\displaystyle \int_{\mathcal{P}}\int_{\mathcal{C}^{-}(X)}F_{0}(X+\tau(Z)V, V)\Phi(X, V, -\tau(Z))\,dVdX \vspace{2mm}\\
\displaystyle+\int_{\mathcal{P}}\int_{\mathcal{C}^{-}(X)}\int_{-\infty}^{-\tau(Z)}V\cdot\nabla_{X} F_{0}(X-tV, V)\Phi(X, V, t)\,dVdX,
\end{array}\tag{T--,--}
\end{equation}
\begin{equation}\label{teetwominus}
\begin{array}{c}
\displaystyle \int_{\mathcal{P}}\int_{\mathcal{C}^{-}(X)}\int_{-\tau(Z)}^{\infty}F(X, V, t)\partial_{t}\Phi(X, V, t)\,dtdVdX \vspace{2mm}\\
=\displaystyle -\int_{\mathcal{P}}\int_{\mathcal{C}^{-}(X)}F_{0}(X+\tau(Z)V, \sigma(X, V)V)\Phi(X, V, -\tau(Z))\,dVdX \vspace{2mm}\\
\displaystyle+\int_{\mathcal{P}}\int_{\mathcal{C}^{-}(X)}\int_{-\tau(Z)}^{\infty}\sigma(X, V)V\cdot\nabla_{X} F_{0}(\sigma(X, V)(X-tV), \sigma(X, V)V)\Phi(X, V, t)\,dVdX,
\end{array}\tag{T--,+}
\end{equation}
\begin{equation}\label{teeoneplus}
\begin{array}{c}
\displaystyle \int_{\mathcal{P}}\int_{\mathcal{C}^{+}(X)}\int_{-\infty}^{-\tau(Z)}F(X, V, t)\partial_{t}\Phi(X, V, t)\,dtdVdX \vspace{2mm}\\
=\displaystyle \int_{\mathcal{P}}\int_{\mathcal{C}^{+}(X)}F_{0}(X+\tau(Z)V, \sigma(X, V)V)\Phi(X, V, -\tau(Z))\,dVdX \vspace{2mm}\\
\displaystyle+\int_{\mathcal{P}}\int_{\mathcal{C}^{+}(X)}\int_{-\infty}^{\tau}\sigma(X, V)V\cdot\nabla_{X} F_{0}(\sigma(X, V)(X-tV), \sigma(X, V)V)\Phi(X, V, t)\,dVdX,
\end{array}\tag{T+,--}
\end{equation}
\begin{equation}\label{teetwoplus}
\begin{array}{c}
\displaystyle \int_{\mathcal{P}}\int_{\mathcal{C}^{+}(X)}\int_{-\tau(Z)}^{\infty}F(X, V, t)\partial_{t}\Phi(X, V, t)\,dtdVdX \vspace{2mm}\\
=-\displaystyle \int_{\mathcal{P}}\int_{\mathcal{C}^{+}(X)}F_{0}(X+\tau(Z)V, V)\Phi(X, V, -\tau(Z))\,dVdX \vspace{2mm}\\
\displaystyle+\int_{\mathcal{P}}\int_{\mathcal{C}^{+}(X)}\int_{-\infty}^{-\tau(Z)}V\cdot\nabla_{X} F_{0}(X-tV, V)\Phi(X, V, t)\,dVdX.
\end{array}\tag{T+,+}
\end{equation}
\end{prop}
With the above identities \eqref{teeoneminus}--\eqref{teetwoplus} in place, we now progress to the more involved evaluation of the space derivative term $J(\Phi)$.
\subsection{The Space Derivative Term $J(\Phi)$}
Whilst the evaluation of the integral $I(\Phi)$ involved little more than integration by parts (alongside tracking the manner in which $T_{t}$ acts on $\mathcal{D}$ as $t$ varies), the evaluation of its spatial counterpart $J(\Phi)$ proves to be more work. We begin as we did above by decomposing the integral $J(\Phi)$ in the following manner:
\begin{equation}\label{refoko}
\begin{array}{c}
J(\Phi)=\displaystyle \underbrace{\int_{\mathcal{P}}\int_{\mathbb{R}^{6}\setminus\mathcal{C}(X)}\int_{-\infty}^{\infty}F(X, V, t)(V\cdot\nabla_{X})\Phi(X, V, t)\,dtdVdX}_{J_{1}(\Phi):=} \vspace{2mm} \\
+\displaystyle \underbrace{\int_{\mathcal{P}}\int_{\mathcal{C}(X)}\int_{-\infty}^{\infty}F(X, V, t)(V\cdot\nabla_{X})\Phi(X, V, t)\,dtdVdX}_{J_{2}(\Phi):=},
\end{array}
\end{equation}
where we further decompose $J_{2}(\Phi)$ as
\begin{equation*}
J_{2}(\Phi)=J_{2}^{(-, -)}(\Phi)+J_{2}^{(-, +)}(\Phi)+J_{2}^{(+, -)}(\Phi)+J_{2}^{(+, +)}(\Phi)
\end{equation*}
in the manner of the above section. Once again, owing to the lack of collisions for those trajectories beginning from data in $Z\in\sqcup_{X\in\mathcal{P}}\mathbb{R}^{6}\setminus\mathcal{C}(X)$, using the fact that $\partial\mathcal{P}$ is a real-analytic submanifold of $\mathbb{R}^{6}$, it follows quickly by the Divergence Theorem for exterior domains that
\begin{equation*}
\begin{array}{l}
\displaystyle J_{1}(\Phi)=\int_{\partial\mathcal{P}}\int_{\mathbb{R}^{6}\setminus\mathcal{C}(X)}\int_{-\infty}^{\infty}F_{0}(Y-tV, V)\Phi(Y, V, t)V\cdot \widehat{\nu}_{\ast}(Y)\,dtdVd\mathscr{H}(Y)\vspace{2mm}\\
\displaystyle-\int_{\mathcal{P}}\int_{\mathbb{R}^{6}\setminus\mathcal{C}(X)}\int_{-\infty}^{\infty}V\cdot\nabla_{X} F_{0}(X-tV, V)\Phi(X, V, t)\,dtdVdX.
\end{array}
\end{equation*}
Let us now consider the case when the limits of integration depend on $X$. For the case of the integral $J_{2}^{(-, -)}(\Phi)$, one finds that
\begin{equation}\label{refo}
\begin{array}{l}
\displaystyle J_{2}^{(-, -)}(\Phi)=\int_{\mathcal{P}}\int_{\mathcal{C}^{-}(X)}\int_{-\infty}^{-\tau(X, V)}\mathrm{div}_{X}\left(F(X, V, t)\Phi(X, V, t)V\right)\,dtdVdX\vspace{2mm}\\
\displaystyle -\int_{\mathcal{P}}\int_{\mathcal{R}^{(-, -)}(X)}V\cdot\nabla_{X} F_{0}(X-tV, V)\Phi(X, V, t)\,d\mu_{X}^{(-, -)}(V, t)dX.
\end{array}
\end{equation}
By applying the Reynolds Transport Theorem to the inner temporal integral of the first term in \eqref{refo} above, one has that
\begin{equation*}
\begin{array}{c}
\displaystyle \int_{\mathcal{P}}\int_{\mathcal{C}^{-}(X)}\int_{-\infty}^{-\tau(X, V)}\mathrm{div}_{X}\left(F(X, V, t)\Phi(X, V, t)V\right)\,dtdVdX \vspace{2mm}\\
\displaystyle = \int_{\mathcal{P}}\int_{\mathcal{C}^{-}(X)}\mathrm{div}_{X}\int_{-\infty}^{-\tau(X, V)}F_{0}(X-tV, V)\Phi(X, V, t)V\,dtdVdX \vspace{2mm}\\
\displaystyle +\int_{\mathcal{P}}\int_{\mathcal{C}^{-}(X)}\left[V\cdot\nabla_{X}\tau(X, V)\right]F_{0}(X+\tau(Z)V, V)\Phi(X, V, -\tau(X, V))\,dVdX \vspace{2mm}\\
= \displaystyle \int_{\mathcal{P}}\int_{\mathcal{C}^{-}(X)}\mathrm{div}_{X}\int_{-\infty}^{-\tau(X, V)}F_{0}(X-tV, V)\Phi(X, V, t)V\,dtdVdX \vspace{2mm}\\
\displaystyle -\int_{\mathcal{P}}\int_{\mathcal{C}^{-}(X)}F_{0}(X+\tau(Z)V, V)\Phi(X, V, -\tau(X, V))\,dVdX
\end{array}
\end{equation*}
by also employing the important transport identity \eqref{transportidentity} above. By applying a similar analysis to the other contributions to $J_{2}(\Phi)$, one obtains the following proposition.
\begin{prop}\label{spatialid}
Suppose $F_{0}\in C^{1}(\mathcal{D})\cap L^{1}(\mathcal{D})$. For any test function $\Phi\in C^{1}_{c}(T\mathbb{R}^{6}\times (-\infty, \infty))$, the map $F$ defined by satisfies the following identities:
\begin{equation}\label{essoneminus}
\begin{array}{c}
\displaystyle \int_{\mathcal{P}}\int_{C^{-}(X)}\int_{-\infty}^{-\tau(X, V)}F(X, V, t)(V\cdot\nabla_{X})\Phi(X, V, t)\,dtdVdX \vspace{2mm}\\
\displaystyle =-\int_{\mathcal{P}}\int_{\mathcal{C}^{-}(X)}F(X+\tau(X, V), V)\Phi(X, V, -\tau(X, V))\,dVdX\vspace{2mm}\\
+\displaystyle \int_{\mathcal{P}}\int_{\mathcal{C}^{-}(X)}\mathrm{div}_{X}\int_{-\infty}^{-\tau(X, V)}F_{0}(X-tV, V)\Phi(X, V, t) V\,dtdVdX \vspace{2mm}\\
-\displaystyle \int_{\mathcal{P}}\int_{\mathcal{C}^{-}(X)}\int_{-\infty}^{-\tau(X, V)}V\cdot\nabla_{X} F_{0}(X-tV, V)\Phi(X, V, t)\,dtdVdX
\end{array}\tag{S--,--}
\end{equation}
\begin{equation}\label{esstwominus}
\begin{array}{c}
\displaystyle 
\int_{\mathcal{P}}\int_{\mathcal{C}^{-}(X)}\int_{-\tau(X, V)}^{\infty}F(X, V, t)(V\cdot\nabla_{X})\Phi(X, V, t)\,dtdVdX \vspace{2mm}\\
\displaystyle =\int_{\mathcal{P}}\int_{C^{-}(X)}F_{0}(X+\tau(X, V)V, \sigma(X, V)V)\Phi(X, V, -\tau(X, V))\,dVdX \vspace{2mm}\\
\displaystyle +\int_{\mathcal{P}}\int_{\mathcal{C}^{-}(X)}\mathrm{div}_{X}\int_{-\tau(X, V)}^{\infty}F_{0}(\sigma(X, V)(X-tV), \sigma(X, V)V)\Phi(X, V, t)V\,dtdVdX\vspace{2mm} \\
\displaystyle -\int_{\mathcal{P}}\int_{\mathcal{C}^{-}(X)}\int_{-\tau(X, V)}^{\infty}\sigma(X, V)V\cdot\nabla_{X} F_{0}(\sigma(X, V)(X-tV), \sigma(X, V)V)\Phi(X, V, t)\,dtdVdX
\end{array}\tag{S--,+}
\end{equation}
\begin{equation}\label{essoneplus}
\begin{array}{c}
\displaystyle \int_{\mathcal{P}}\int_{C^{+}(X)}\int_{-\infty}^{-\tau(X, V)}F(X, V, t)(V\cdot\nabla_{X})\Phi(X, V, t)\,dtdVdX \vspace{2mm}\\
\displaystyle =-\int_{\mathcal{P}}\int_{\mathcal{C}^{+}(X)}F(X+\tau(X, V), \sigma(X, V)V)\Phi(X, V, -\tau(X, V))\,dVdX\vspace{2mm}\\
+\displaystyle \int_{\mathcal{P}}\int_{\mathcal{C}^{+}(X)}\mathrm{div}_{X}\int_{-\infty}^{-\tau(X, V)}F_{0}(\sigma(X, V)(X-tV), \sigma(X, V)V)\Phi(X, V, t) V\,dtdVdX \vspace{2mm}\\
-\displaystyle \int_{\mathcal{P}}\int_{\mathcal{C}^{-}(X)}\int_{-\infty}^{-\tau(X, V)}\sigma(X, V)V\cdot\nabla_{X} F_{0}(\sigma(X, V)(X-tV), \sigma(X, V)V)\Phi(X, V, t)\,dtdVdX
\end{array}\tag{S+,--}
\end{equation}
\begin{equation}\label{esstwoplus}
\begin{array}{c}
\displaystyle 
\int_{\mathcal{P}}\int_{\mathcal{C}^{+}(X)}\int_{-\tau(X, V)}^{\infty}F(X, V, t)(V\cdot\nabla_{X})\Phi(X, V, t)\,dtdVdX \vspace{2mm}\\
\displaystyle =\int_{\mathcal{P}}\int_{C^{+}(X)}F_{0}(X+\tau(X, V)V, V)\Phi(X, V, -\tau(X, V))\,dVdX \vspace{2mm}\\
\displaystyle +\int_{\mathcal{P}}\int_{\mathcal{C}^{+}(X)}\mathrm{div}_{X}\int_{-\tau(X, V)}^{\infty}F_{0}(X-tV, V)\Phi(X, V, t)V\,dtdVdX\vspace{2mm} \\
\displaystyle -\int_{\mathcal{P}_{g}}\int_{\mathcal{C}^{-}(X)}\int_{-\tau(X, V)}^{\infty}V\cdot\nabla_{X} F_{0}(X-tV, V)\Phi(X, V, t)\,dtdVdX
\end{array}\tag{S+,+}
\end{equation}
\end{prop}
At this point it is not necessary to apply the Reynolds Transport Theorem once again to pass the divergence operator through those integrals over $X$-dependent velocity cones above. Indeed, as we shall see below, this operation becomes redundant once we sum the contributions of $I(\Phi)$ and $J(\Phi)$.
\subsection{Combining the Contributions from $I(\Phi)$ and $J(\Phi)$}
We now seek for cancellations that arise when the temporal integral \eqref{iint} and the spatial integral \eqref{jint} are summed, thereby proving the global-in-time existence of weak solutions of the Liouville equation in the case that the initial datum is taken to be $F_{0}\in C^{1}(\mathcal{D})\cap L^{1}(\mathcal{D})$. One finds by matching \eqref{teeoneminus} with \eqref{essoneminus} (and similarly with the other analogous terms of propositions \ref{temporalid} and \ref{spatialid} above) that
\begin{equation*}
I(\Phi)+J(\Phi)=\int_{\mathcal{P}}\int_{\mathbb{R}^{6}}\mathrm{div}_{X}\int_{-\infty}^{\infty}F(X, V, t)\Phi(X, V, t)V\,dtdVdX
\end{equation*}
for all $\Phi\in C^{1}_{c}(T\mathbb{R}^{6}\times (-\infty, \infty))$. We note that one may now simply pass the divergence operator through the velocity integral over $\mathbb{R}^{6}$ by means of nothing more than the Dominated Convergence Theorem, yielding
\begin{equation*}
I(\Phi)+J(\Phi)=-\int_{\partial\mathcal{P}}\int_{\mathbb{R}^{6}}\int_{-\infty}^{\infty}F(Y, V, t)\Phi(Y, V, t)V\cdot\widehat{\nu}_{\ast}(Y)\,dtdVd\mathscr{H}(Y)
\end{equation*}
following one final application of the Divergence Theorem in exterior domains. In other words, we have demonstrated -- by appealing only to the definition of the hard sphere flow $\{T_{t}\}_{t\in\mathbb{R}}$ -- that for smooth initial distribution functions $F_{0}\in C^{1}(\mathcal{D})$, the mild solution
\begin{equation*}
F(X, V, t)=F_{0}((\Pi_{1}\circ T_{-t})Z, (\Pi_{2}\circ T_{-t})Z)
\end{equation*}
for $Z=[X, V]\in\mathcal{D}$ is a global-in-time weak solution of the Liouville equation
\begin{equation*}
\partial_{t}F+(V\cdot\nabla_{X})F=C[F] \quad \text{on}\hspace{2mm}\mathcal{D}\times (-\infty, \infty),
\end{equation*}
where $C:L^{1}(\mathcal{D}\times (-\infty, \infty))\rightarrow\mathscr{S}'$ is the {\bf collision operator} associated to the Liouville dynamics on $\mathcal{D}$ defined by
\begin{equation}\label{colly}
\langle C[F], \Phi \rangle_{\mathscr{S}'\times\mathscr{S}}:=\int_{\partial\mathcal{P}}\int_{\mathbb{R}^{6}}\int_{-\infty}^{\infty}F(Y, V, t)\Phi(Y, V, t)\underbrace{V\cdot\widehat{\nu}_{\ast}(Y)}_{b_{\ast}(Z):=}\,dtdVd\mathscr{H}(Y)
\end{equation}
for $F\in L^{1}(\mathcal{D}\times (-\infty, \infty))$ and $\Phi\in\mathscr{S}$, where we call $b_{\ast}:\partial\mathcal{P}\times\mathbb{R}^{6}\rightarrow\mathbb{R}$ the generalised {\bf scattering cross-section}. As has been discussed above, the evaluation of the integral \eqref{colly} involves one taking a trace of the distribution function $F$ on a measure zero subset of $\mathcal{D}\times (-\infty, \infty)$. One would ideally like to be able to make sense of a solution of the Liouville equation for which the initial datum $F_{0}$ can be taken from a class more general than simply that of everywhere pointwise defined integrable functions. Although straightforward, this shall not be done here. 
\section{A New Method of Construction}\label{free}
In this section, our `new' method of construction of weak solutions of the Liouville equation \eqref{weakdude} follows from the construction of a free transport dynamics on an auxiliary smooth manifold, followed by a `folding' procedure back onto the original hard sphere phase space $\mathcal{D}$. We now formalise the rough discussion on the Sinai billiard in the introduction above for the case of the hard sphere table $\mathcal{P}$.  
\begin{rem}
The following approach may be recast in the language of symplectic geometry by considering the adjunction table as an adjunction manifold, and by considering solutions as integral curves of the complete vector field generated by the kinetic energy Hamiltonian on the associated tangent bundle (see, for instance, \textsc{Lee} \cite{lee2013introduction}).
\end{rem}
\subsection{A Formal Identification}\label{formaliddy}
We begin with an exploratory discussion. From our work above, we know the explicit form of the hard sphere trajectories $t\mapsto Z(t):=T_{t}Z_{0}$ for every $Z_{0}\in\mathcal{D}$. This explicit formula is useful for the purposes of seeing how one might understand the hard sphere trajectories $t\mapsto T_{t}Z_{0}$ as {\em smooth} on some larger auxiliary set. In this direction, let us consider initial data $X_{0}\in\mathcal{P}^{\circ}$ and $V_{0}\in\mathcal{C}^{-}(X_{0})$ and study the associated trajectory in $\mathcal{D}$. The spatial map $t\mapsto X(t):=(\Pi_{1}\circ T_{t})Z_{0}$ is given by
\begin{equation*}
X(t):=\left\{
\begin{array}{ll}
X_{0}+tV_{0} & \quad \text{if} \hspace{2mm}-\infty<t\leq \tau(X_{0}, V_{0}),\vspace{2mm}\\
X_{0}+\tau(X_{0}, V_{0})V_{0}+(t-\tau(X_{0}, V_{0}))\sigma(X_{0}, V_{0})V_{0} &\quad \text{if}\hspace{2mm}\tau(X_{0}, V_{0})<t<\infty.
\end{array}
\right.
\end{equation*}
However, using basic properties of the Boltzmann scattering matrices, it follows that $X$ may be rewritten in the form
\begin{equation*}
X(t):=\left\{
\begin{array}{ll}
X_{0}+tV_{0} & \quad \text{if} \hspace{2mm}-\infty<t\leq \tau(X_{0}, V_{0}),\vspace{2mm}\\
\sigma(X_{0}, V_{0})\left(r[X_{0}+\tau(X_{0}, V_{0})V_{0}]+(t-\tau(X_{0}, V_{0}))V_{0}\right) &\quad \text{if}\hspace{2mm}\tau(X_{0}, V_{0})<t<\infty,
\end{array}
\right.
\end{equation*}
where $r:\mathbb{R}^{6}\rightarrow\mathbb{R}^{6}$ is the {\em switching operator} defined by
\begin{equation*}
r(X):=\left[
\begin{array}{c}
\ov{x} \\
x
\end{array}
\right]
\end{equation*}
for $X=[x, \ov{x}]\in\mathbb{R}^{6}$. Let us suppose {\em formally} that the point $r(X_{0}+\tau(Z_{0})V_{0})\in\partial\mathcal{P}$ is identified with the point $X_{0}+\tau(Z_{0})V_{0}\in\partial\mathcal{P}$. We denote the associated quotient space by $\mathcal{P}/\sim$. It is readily seen that the map $X:\mathbb{R}\rightarrow\mathcal{P}$ may be rewritten as a map $X_{\ast}$ with range in the associated quotient space $\mathcal{P}/\sim$ given by
\begin{equation*}
X_{\ast}(t):=\left\{
\begin{array}{ll}
X_{0}+tV_{0} & \quad \text{if} \hspace{2mm}-\infty<t\leq \tau(X_{0}, V_{0}),\vspace{2mm}\\
\sigma(X_{0}, V_{0})(X_{0}+tV_{0}) &\quad \text{if}\hspace{2mm}\tau(X_{0}, V_{0})<t<\infty.
\end{array}
\right.
\end{equation*}
Noting that $X_{\ast}$ is formally continuous at $t=\tau(X_{0}, V_{0})$, it follows that if we consider the trajectory $t\mapsto X_{\ast}(t)$ as passing from one copy of $\mathcal{P}/\sim$ to another copy thereof which is `globally twisted' by the Boltzmann scattering matrix $\sigma(X_{0}, V_{0})$, we deduce that $X_{\ast}$ may be viewed as a `straight line' trajectory $X_{\ast\ast}$ given by
\begin{equation*}
X_{\ast\ast}(t)=X_{0}+tV_{0} \quad \text{for all}\hspace{2mm}t\in(-\infty, \infty).
\end{equation*}
In particular, when viewed in this light there is no change in the velocity of this trajectory $X_{\ast\ast}$ (unlike for $X$), and so hard sphere trajectories appear {\em formally} as straight lines on some quotient space. More generally, the above example suggests that if we identify the point $Y\in\partial\mathcal{P}$ on one copy of the table $\mathcal{P}$ with the point $r(Y)\in\partial\mathcal{P}$ of another copy, then one might try to study hard sphere dynamics by eliminating collisions entirely. 
\subsection{Some Notation}
As we define the disjoint union $\mathcal{P}\sqcup\mathcal{P}$ to be
\begin{equation*}
\mathcal{P}\sqcup\mathcal{P}:=\left\{
(X, j)\,:\,X\in\mathcal{P}\hspace{2mm}\text{and}\hspace{2mm}j=1, 2
\right\},
\end{equation*}
we define the associated subsets $(\mathcal{P}\sqcup\mathcal{P})_{j}$ thereof by
\begin{equation*}
(\mathcal{P}\sqcup\mathcal{P})_{j}:=\left\{
(X, j)\,:\,X\in\mathcal{P}
\right\}.
\end{equation*}
Moreover, we write points $\zeta\in\mathcal{P}\sqcup\mathcal{P}\times\mathbb{R}^{6}$ by $[X, V]_{j}$ for $X\in\mathcal{P}$, $V\in\mathbb{R}^{6}$ and $j=1, 2$. Moreover, we use the notation $\iota_{j}:\mathcal{D}\rightarrow\mathcal{P}\sqcup\mathcal{P}\times\mathbb{R}^{6}$ to denote the natural inclusion operators given by
\begin{equation*}
\iota_{j}([X, V]):=[X, V]_{j}\quad \text{for}\hspace{2mm}[X, V]\in\mathcal{D}.
\end{equation*}
Finally, we write $\mathscr{S}(T\mathbb{R}^{N}, \mathbb{R}^{M})$ to denote the Fr\'{e}chet space of $\mathbb{R}^{M}$-valued Schwartz maps on $\mathbb{R}^{N}$, and $\mathscr{S}(T\mathbb{R}^{N}, \mathbb{R}^{M})'$ to denote the associated space of tempered distributions: see \textsc{Friedlander and Joshi} \cite{friedlander1998introduction}. When $M=1$, we write $\mathscr{S}(T\mathbb{R}^{N}, \mathbb{R}^{M})$ simply as $\mathscr{S}(T\mathbb{R}^{N})$.
\subsection{Definition of a (Hamiltonian) Flow on $\mathcal{P}\sqcup\mathcal{P}\times\mathbb{R}^{6}$}
We begin with the definition of a flow operator on $\mathcal{P}\sqcup\mathcal{P}\times\mathbb{R}^{6}$ which is inspired by the formal discussion of section \ref{formaliddy} above.
\begin{defn}[Hamiltonian Flow on $\mathcal{P}\sqcup\mathcal{P}\times\mathbb{R}^{6}$]
For each time $t\in\mathbb{R}$, we write $S_{t}:\mathcal{P}\sqcup\mathcal{P}\times\mathbb{R}^{6}\rightarrow \mathcal{P}\sqcup\mathcal{P}\times\mathbb{R}^{6}$ to denote the operator which is defined pointwise by
\begin{equation*}
S_{t}\zeta:=\left\{
\begin{array}{ll}
[X+tV, V]_{1} & \quad \text{if}\hspace{2mm}-\infty<t\leq \tau(X, V), \vspace{2mm}\\

[r[X+\tau(X, V)V]+(t-\tau(X, V))V, V]_{2} & \quad \text{if} \hspace{2mm}\tau(X, V)<t<\infty,
\end{array}
\right.
\end{equation*}
if $\zeta$ is such that $X\in(\mathcal{P}\sqcup\mathcal{P})_{1}$ with $V\in\mathcal{C}^{-}(X)$; and 
\begin{equation*}
S_{t}\zeta:=\left\{
\begin{array}{ll}
[r[X+\tau(X, V)V]+(t-\tau(X, V))V, V]_{2} & \quad \text{if} \hspace{2mm}-\infty<t\leq \tau(X, V), \vspace{2mm}\\

[X+tV, V]_{1} & \quad \text{if}\hspace{2mm}\tau(X, V)<t<\infty,
\end{array}
\right.
\end{equation*}
when $\zeta$ is such that $X\in(\mathcal{P}\sqcup\mathcal{P})_{1}$ with $V\in\mathcal{C}^{+}(X)$; and
\begin{equation*}
S_{t}\zeta:=[X+tV, V]_{1} \quad \text{for all}\hspace{2mm}t\in(-\infty, \infty)
\end{equation*}
when $\zeta$ is such that $X\in(\mathcal{P}\sqcup\mathcal{P})_{1}$ with $V\in \mathbb{R}^{6}\setminus\mathcal{C}^{-}(X)$. 
Finally, \begin{equation*}
S_{t}\zeta:=\left\{
\begin{array}{ll}
[X+tV, V]_{1} & \quad \text{if}\hspace{2mm}-\infty<t\leq \tau(X, V), \vspace{2mm}\\

[r[X+\tau(X, V)V]+(t-\tau(X, V))V, V]_{1} & \quad \text{if} \hspace{2mm}\tau(X, V)<t<\infty,
\end{array}
\right.
\end{equation*}
if $\zeta$ is such that $X\in(\mathcal{P}\sqcup\mathcal{P})_{2}$ with $V\in\mathcal{C}^{-}(X)$; and 
\begin{equation*}
S_{t}\zeta:=\left\{
\begin{array}{ll}
[r[X+\tau(X, V)V]+(t-\tau(X, V))V, V]_{1} & \quad \text{if} \hspace{2mm}-\infty<t\leq \tau(X, V), \vspace{2mm}\\

[X+tV, V]_{2} & \quad \text{if}\hspace{2mm}\tau(X, V)<t<\infty,
\end{array}
\right.
\end{equation*}
when $\zeta$ is such that $X\in(\mathcal{P}\sqcup\mathcal{P})_{2}$ with $V\in\mathcal{C}^{+}(X)$; and
\begin{equation*}
S_{t}\zeta:=[X+tV, V]_{2} \quad \text{for all}\hspace{2mm}t\in(-\infty, \infty)
\end{equation*}
when $\zeta$ is such that $X\in(\mathcal{P}\sqcup\mathcal{P})_{2}$ with $V\in \mathbb{R}^{6}\setminus\mathcal{C}^{-}(X)$.
\end{defn}
It is clear from its definition that the flow $\{S_{t}\}_{t\in\mathbb{R}}$ corresponds to a kind of free transport on phase space $\mathcal{P}\sqcup\mathcal{P}\times\mathbb{R}^{6}$, if we consider the points $[X, V]_{1}$ and $[r[X], V]_{2}$ as formally identified. As such, we think of this flow as being associated with the formal transport equation
\begin{equation}\label{freetranny}
\frac{\partial \mathfrak{f}}{\partial t}+P\cdot\nabla_{Q}\mathfrak{f}=0 \tag{FT}
\end{equation}
on a suitable quotient of $\mathcal{P}\sqcup\mathcal{P}\times\mathbb{R}^{6}$. Of course, at the moment, as we have not endowed the doubling $\mathcal{P}\sqcup\mathcal{P}$ of the hard sphere table with a differentiable structure, the derivatives in \eqref{freetranny} have no real meaning. However, this causes us no problems if we work with the following notion of mild solution of the free transport equation. Indeed, we establish the following definition which is in line with \textsc{Cercignani, Illner and Pulvirenti} (\cite{cercignani2013mathematical}, page 69) in the case of the formal Liouville equation for hard spheres.
\begin{defn}[Mild Solution of \eqref{freetranny}]
We say that $\mathfrak{f}$ is a global-in-time {\bf mild solution} of the free transport equation \eqref{freetranny} on $\mathcal{P}\sqcup\mathcal{P}\times\mathbb{R}^{6}$ subject to the initial datum $\mathfrak{f}_{0}:\mathcal{P}\sqcup\mathcal{P}\times\mathbb{R}^{6}\rightarrow\mathbb{R}$ if and only if
\begin{equation*}
\mathfrak{f}(\zeta, t)=\mathfrak{f}_{0}(S_{-t}\zeta).
\end{equation*}
\end{defn}
\begin{rem}
We claim that for each $\zeta\in \mathcal{P}\sqcup\mathcal{P}\times\mathbb{R}^{6}$, the associated global trajectory
\begin{equation*}
\left\{
S_{t}\zeta\in \mathcal{P}\sqcup\mathcal{P}\times\mathbb{R}^{6}\,:\,t\in(-\infty, \infty)
\right\}
\end{equation*}
can be seen as the integral curve of a Hamiltonian vector field on $T^{\ast}\mathcal{M}$, where $\mathcal{M}$ denotes an adjunction manifold associated to the disjoint union $\mathcal{P}\sqcup\mathcal{P}$. We develop this idea further for $N\geq 2$ hard spheres in the forthcoming work \textsc{Wilkinson} \cite{me1}. For the moment, however, we shall treat the flow $\{S_{t}\}_{t\in\mathbb{R}}$ as an auxiliary object which allows us to understand the corruption of chaos for global-in-time weak solutions of the Liouville equation \eqref{weakdude} subject to chaotic initial data.
\end{rem}
\subsection{Connecting Mild Solutions of \eqref{freetranny} with Weak Solutions of \eqref{weakdude}}
For any $\mathfrak{f}_{0}$ with the property that $Z\mapsto \mathfrak{f}_{0}\circ\iota_{j}(Z)$ is of class $L^{1}_{\mathrm{loc}}$ on $\mathcal{D}$ for $j=1, 2$, the associated mild solution $\mathfrak{f}$ of \eqref{freetranny} gives rise to a continuous linear functional on Schwartz space $\mathscr{S}(T\mathbb{R}^{6}, \mathbb{R}^{2})$ by the natural action
\begin{align*}
\langle \mathfrak{f}, \Psi\rangle:=\frac{1}{2}\int_{\mathcal{P}}\int_{\mathbb{R}^{6}}\int_{-\infty}^{\infty}\mathfrak{f}_{0}\circ S_{-t}\circ \iota_{1}(X, V, t)\Psi_{1}(X, V, t)\,dtdVdX\vspace{2mm}\\
+\frac{1}{2}\int_{\mathcal{P}}\int_{\mathbb{R}^{6}}\int_{-\infty}^{\infty}\mathfrak{f}_{0}\circ S_{-t}\circ \iota_{2}(X, V, t)\Psi_{2}(X, V, t)\,dtdVdX,
\end{align*}
for $\Psi=(\Psi_{1}, \Psi_{2})\in\mathscr{S}(T\mathbb{R}^{6}, \mathbb{R}^{2})$. Exploiting smoothness of the Boltzmann scattering maps $\sigma_{\ast}$, we define the operator $R:\mathscr{S}(T\mathbb{R}^{6})\rightarrow\mathscr{S}(T\mathbb{R}^{6}, \mathbb{R}^{2})$ by
\begin{equation*}
(R\Phi)(X, V, t):=\left(
\begin{array}{c}
\Phi(X, V, t) \vspace{2mm} \\
\Phi(\sigma_{\ast}(X, V)X, \sigma_{\ast}(X, V)V, t)
\end{array}
\right).
\end{equation*}
Suppose we endow the free transport equation \eqref{freetranny} with initial data $\mathfrak{f}_{0}$ of the shape
\begin{equation*}
\mathfrak{f}_{0}(\zeta):=\left\{
\begin{array}{ll}
F_{0}(X, V) & \quad \text{if}\hspace{2mm}\zeta=[X, V]_{1}\hspace{2mm}\text{for some}\hspace{2mm}[X, V]\in\mathcal{D}, \vspace{2mm}\\
G_{0}(X, V) & \quad \text{if}\hspace{2mm}\zeta=[X, V]_{2}\hspace{2mm}\text{for some}\hspace{2mm}[X, V]\in\mathcal{D},
\end{array}
\right.
\end{equation*}
where $G_{0}(X, V):=F_{0}(\sigma_{\ast}(X, V)X, \sigma_{\ast}(X, V)V)$. It then follows by a simple decomposition argument that
\begin{align*}
\langle \mathfrak{f}, R\Phi \rangle=\frac{1}{2}\underbrace{\int_{\mathcal{P}}\int_{C^{-}(X)}\int_{-\infty}^{-\tau(X, V)}G_{0}(r[X+\tau V]-(t+\tau)V, V)\Phi(X, V, t)\,dtdVdX}_{K^{(-, -)}(\Phi):=}\vspace{2mm}\\
+\frac{1}{2}\underbrace{\int_{\mathcal{P}}\int_{C^{-}(X)}\int_{-\tau(X, V)}^{\infty}F_{0}(X-tV, V)\Phi(X, V, t)\,dtdVdX}_{K^{(-, +)}(\Phi):=} \vspace{2mm} \\
+\frac{1}{2}\underbrace{\int_{\mathcal{P}}\int_{C^{+}(X)}\int_{-\infty}^{-\tau(X, V)}F_{0}(X-tV, V)\Phi(X, V, t)\,dtdVdX}_{K^{(+, -)}(\Phi):=}\vspace{2mm}\\
\frac{1}{2}\underbrace{\int_{\mathcal{P}}\int_{C^{+}(X)}\int_{-\tau(X, V)}^{\infty}G_{0}(r[X+\tau V]-(t+\tau)V, V)\Phi(X, V, t)\,dtdVdX}_{K^{(+, +)}(\Phi):=}\vspace{2mm}\\
+\underbrace{\int_{\mathcal{P}}\int_{\mathbb{R}^{6}}\int_{-\infty}^{\infty}F_{0}(X-tV, V)\Phi(X, V, t)\,dtdVdX}_{K^{(0)}(\Phi):=}\vspace{2mm}\\
\frac{1}{2}\underbrace{\int_{\mathcal{P}}\int_{C^{-}(X)}\int_{-\infty}^{-\tau(X, V)}F_{0}(r[X+\tau V]-(t+\tau)V, V)\Phi(\sigma(X, V)X, \sigma(X, V)V, t)\,dtdVdX}_{L^{(-, -)}(\Phi):=}\vspace{2mm}\\
+\frac{1}{2}\underbrace{\int_{\mathcal{P}}\int_{C^{-}(X)}\int_{-\tau(X, V)}^{\infty}G_{0}(X-tV, V)\Phi(\sigma(X, V)X, \sigma(X, V)V, t)\,dtdVdX}_{L^{(-, +)}(\Phi):=} \vspace{2mm} \\
+\frac{1}{2}\underbrace{\int_{\mathcal{P}}\int_{C^{+}(X)}\int_{-\infty}^{-\tau(X, V)}G_{0}(X-tV, V)\Phi(\sigma(X, V)X, \sigma(X, V)V, t)\,dtdVdX}_{L^{(+, -)}(\Phi):=}\vspace{2mm}\\
\frac{1}{2}\underbrace{\int_{\mathcal{P}}\int_{C^{+}(X)}\int_{-\tau(X, V)}^{\infty}F_{0}(r[X+\tau V]-(t+\tau)V, V)\Phi(\sigma(X, V)X, \sigma(X, V)V, t)\,dtdVdX}_{L^{(+, +)}(\Phi):=}\vspace{2mm}\\
+\underbrace{\int_{\mathcal{P}}\int_{\mathbb{R}^{6}}\int_{-\infty}^{\infty}G_{0}(X-tV, V)\Phi(\sigma(X, V)X, \sigma(X, V)V, t)\,dtdVdX}_{L^{(0)}(\Phi):=}.
\end{align*}
Using the fact that 
\begin{equation*}
\sigma_{\ast}(X, V)r[X+\tau(X, V)V]=X+\tau(X, V)V,
\end{equation*}
it follows that
\begin{align*}
\langle \mathfrak{f}, R\Phi\rangle=\frac{1}{2}\int_{\mathcal{P}}\int_{\mathbb{R}^{6}}\int_{-\infty}^{\infty}F(X, V, t)\Phi(X, V, t)\,dtdVdX\vspace{2mm}\\
\frac{1}{2}\int_{\mathcal{P}}\int_{\mathbb{R}^{6}}\int_{-\infty}^{\infty}F(\sigma_{\ast}(X, V)X, \sigma_{\ast}(X, V)V, t)\Phi(\sigma_{\ast}(X, V)X, \sigma_{\ast}(X, V)V, t)\,dtdVdX.
\end{align*}
Finally, by recalling that the map $\Sigma_{\ast}:\mathcal{D}\rightarrow\mathcal{D}$ is a smooth measure-preserving diffeomorphism, we deduce that
\begin{equation*}
\langle \mathfrak{f}, R\Phi\rangle=\langle F, \Phi\rangle,
\end{equation*}
and so we conclude that $F=R'\mathfrak{f}$ in $\mathscr{S}(T\mathbb{R}^{6})'$, namely that the unique weak solution of the Liouville equation \eqref{weakdude} can be seen as the image under the adjoint operator $R':\mathscr{S}(T\mathbb{R}^{6}, \mathbb{R}^{2})'\rightarrow\mathscr{S}(T\mathbb{R}^{6})'$ of the mild solution $\mathfrak{f}$ of the free transport equation \eqref{freetranny}. It is this connection that allows us to understand the manner in which the dynamics of the Liouville equation for hard spheres `corrupts' the structure of chaotic initial data.
\subsection{Some Extension Lemmas}
As we work not just on the hard sphere table $\mathcal{P}$ but on the whole Euclidean space $\mathbb{R}^{6}$ in this section, we need to establish what we mean by the velocity collision cone $\mathcal{C}(Q)$ when $Q\in\mathbb{R}^{6}\setminus\mathcal{P}$.
\begin{defn}[Velocity Collision Cone on $\mathbb{R}^{6}\setminus\mathcal{P}$]
For $Q\in\mathbb{R}^{6}\setminus\mathcal{P}$, we define $\mathcal{C}(Q):=\mathbb{R}^{6}$. Moreover, we write $\mathcal{C}^{-}(Q)$ to denote the subset
\begin{equation*}
\mathcal{C}^{-}(Q):=\left\{
P\in\mathbb{R}^{6}\,:\,(\ov{q}-q)\cdot(\ov{p}-p)\geq 0
\right\}
\end{equation*}
and $\mathcal{C}^{+}(Q)$ to denote
\begin{equation*}
\mathcal{C}^{+}(Q):=\left\{
P\in\mathbb{R}^{6}\,:\,(\ov{q}-q)\cdot(\ov{p}-p)\leq 0
\right\}.
\end{equation*}
\end{defn}
As such, for $Q\in\mathbb{R}^{6}\setminus\mathcal{P}$ the velocity collision `cone' $\mathcal{C}(Q)$ is not really a cone, but rather the whole tangent space $T_{Q}\mathbb{R}^{6}\simeq\mathbb{R}^{6}$. With this in place, we put forth the following extension lemma.
\begin{lem}
The collision time map $\tau:\sqcup_{X\in\mathcal{P}}\mathcal{C}(X)$ admits a smooth extension to the bundle $\sqcup_{Q\in\mathbb{R}^{6}}\mathcal{C}(Q)$.
\end{lem}
\begin{proof}
By employing ideas similar to those used in the proof of proposition \ref{colltimeprop} above, if $Q\in\mathbb{R}^{6}\setminus\mathcal{P}$ is not the zero vector and $D(Q, P)$ denotes the signed distance from $Q$ to the point of the set $L(Q, P)\cap\partial\mathcal{P}$ which is closest to $Q$ in the Euclidean metric, we have that
\begin{equation*}
D(Q, P):=\left\{
\begin{array}{ll}
-(\ov{q}-q)\cdot(\widehat{\ov{p}-p})+\sqrt{|(\ov{q}-q)\cdot(\widehat{\ov{p}-p})|^{2}+\varepsilon^{2}-|\ov{q}-q|^{2}} & \quad \text{if}\hspace{2mm}P\in\mathcal{C}^{-}(Q), \vspace{2mm}\\
-(\ov{q}-q)\cdot(\widehat{\ov{p}-p})-\sqrt{|(\ov{q}-q)\cdot(\widehat{\ov{p}-p})|^{2}+\varepsilon^{2}-|\ov{q}-q|^{2}} & \quad \text{if}\hspace{2mm}P\in\mathcal{C}^{+}(Q).
\end{array}
\right.
\end{equation*}
and so we define
\begin{equation*}
\tau(Q, P):=\frac{D(Q, P)}{|\ov{p}-p|}
\end{equation*}
for $Q\in\mathbb{R}^{6}\setminus\mathcal{P}$. We also define $\tau(0, P):=\varepsilon/|\ov{p}-p|$. One may check that $\tau$ so defined is smooth.
\end{proof}
With the extension of the collision time map in place, this allows us to define velocity-time regions
\begin{equation*}
\begin{array}{c}
\mathcal{R}^{(-, -)}(Q):=\left\{
(V, t)\in\mathbb{R}^{6}\times (-\infty, \infty)\,:\, -t\leq \tau(Q, P)
\right\}, \vspace{2mm}\\
\mathcal{R}^{(-, +)}(Q):=\left\{
(V, t)\in\mathbb{R}^{6}\times (-\infty, \infty)\,:\, -t>\tau(Q, P)
\right\}, \vspace{2mm}\\
\mathcal{R}^{(+, -)}(Q):=\left\{
(V, t)\in\mathbb{R}^{6}\times (-\infty, \infty)\,:\, -t\leq\tau(Q, P)
\right\}, \vspace{2mm}\\
\mathcal{R}^{(+, +)}(Q):=\left\{
(V, t)\in\mathbb{R}^{6}\times (-\infty, \infty)\,:\, -t>\tau(Q, P)
\right\}.
\end{array}
\end{equation*}
We also establish the following extension lemma.
\begin{lem}
The scattering map $\sigma_{\ast}:\sqcup_{X\in\mathcal{P}}\mathcal{V}(X)\rightarrow\mathbb{R}^{6}$ admits a smooth extension to $\sqcup_{Q\in\mathbb{R}^{6}}\mathcal{V}(Q)$. Moreover, the associated map
\begin{equation}\label{siggy}
\Sigma: Z\mapsto\left(
\begin{array}{cc}
\sigma_{\ast}(Q, P) & 0 \\
0 & \sigma_{\ast}(Q, P)
\end{array}
\right)Z
\end{equation}
is a smooth diffeomorphism of $T\mathbb{R}^{6}$ with the property that 
\begin{equation*}
\begin{array}{c}
\Sigma_{\ast}(\mathcal{D}^{\circ})=\mathcal{D}^{\circ},\\
\Sigma_{\ast}(\partial\mathcal{D})=\partial\mathcal{D}, \\
\Sigma_{\ast}((\mathbb{R}^{6}\setminus\mathcal{P})\times\mathbb{R}^{6})=(\mathbb{R}^{6}\setminus\mathcal{P})\times\mathbb{R}^{6}.
\end{array}
\end{equation*}
\end{lem}
\begin{proof}
Using the extension of $\tau$ to $\sqcup_{Q\in\mathbb{R}^{6}}\mathcal{C}(Q)$, we define $\widehat{\nu}:T\mathbb{R}^{6}\rightarrow\mathbb{S}^{5}$ by
\begin{equation*}
\widehat{\nu}(Q, P):=\frac{1}{\sqrt{2}\varepsilon}\left[
\begin{array}{c}
\ov{q}-q+\tau(Q, P)(\ov{p}-p) \\
q-\ov{q}+\tau(Q, P)(p-\ov{p})
\end{array}
\right]
\end{equation*}
and the associated Boltzmann scattering map $\sigma_{\ast}: \sqcup_{Q\in\mathbb{R}^{6}\setminus\mathcal{P}}\mathcal{C}(Q)\rightarrow\mathrm{O}(6)$ by
\begin{equation*}
\sigma_{\ast}(Q, P):=I-2\widehat{\nu}(Q, P)\otimes\widehat{\nu}(Q, P).
\end{equation*}
One may verify that $\sigma_{\ast}$ is a smooth map on $T\mathbb{R}^{6}$. Moreover, one may also verify that $\Sigma_{\ast}$ defined by \eqref{siggy} above is a smooth diffeomorphism with unit Jacobian. It remains to show that $\Sigma_{\ast}$ acts in an invariant manner on the sets $\mathcal{D}^{\circ}$, $\partial\mathcal{D}$ and $(\mathbb{R}^{6}\setminus\mathcal{P})\times\mathbb{R}^{6}$. 

We note first that the map acts $\Sigma_{\ast}$ in an invariant manner on the sets $\mathcal{D}^{\circ}$, $\partial\mathcal{D}$ and $(\mathbb{R}^{6}\setminus\mathcal{P})\times\mathbb{R}^{6}$ if and only if $\sigma_{\ast}(\mathcal{P}^{\circ})=\mathcal{P}^{\circ}$, $\sigma_{\ast}(\partial\mathcal{P})=\partial\mathcal{P}$ and $\sigma_{\ast}(\mathbb{R}^{6}\setminus\mathcal{P})=\mathbb{R}^{6}\setminus\mathcal{P}$, respectively. To demonstrate the latter, let us define the maps $x':T\mathbb{R}^{6}\rightarrow\mathbb{R}^{3}$ and $\ov{x}':T\mathbb{R}^{6}\rightarrow\mathbb{R}^{3}$ by
\begin{equation*}
q'(Q, P):=\left(
\begin{array}{c}
(\sigma(Q, P)Q)_{1} \\
(\sigma(Q, P)Q)_{2} \\
(\sigma(Q, P)Q)_{3}
\end{array}
\right)\quad \text{and}\quad\ov{q}'(Q, P):=\left(
\begin{array}{c}
(\sigma(Q, P)Q)_{4} \\
(\sigma(Q, P)Q)_{5} \\
(\sigma(Q, P)Q)_{6}
\end{array}
\right),
\end{equation*}
one has that
\begin{equation*}
\begin{array}{c}
|q'(Q, P)-\ov{q}'(Q, P)| \vspace{2mm}\\
=|q-\ov{q}-2[(q-\ov{q})\cdot \widehat{n}(Q, P)]\widehat{n}(Q, P)|\vspace{2mm}\\
=|(I-2\widehat{n}(Q, P)\otimes \widehat{n}(Q, P))(q-\ov{q})|,
\end{array}
\end{equation*}
where $I$ denotes here the identity matrix in $\mathbb{R}^{3\times 3}$ and $\widehat{n}: T\mathbb{R}^{6}\rightarrow\mathbb{S}^{2}$ denotes the map
\begin{equation*}
\widehat{n}(Q, P):=\frac{1}{\varepsilon}\frac{\ov{q}-q+\tau(Q, P)(\ov{p}-p)}{|\ov{q}-q+\tau(Q, P)(\ov{p}-p)|}.
\end{equation*}
Thus, as reflection matrices in $\mathrm{O}(3)$ act isometrically on $\mathbb{R}^{3}$, it follows that
\begin{equation*}
|q'(Q, P)-\ov{q}'(Q, P)|=|q-\ov{q}|,
\end{equation*}
from which the claim of the lemma follows.
\end{proof}
\subsection{The Chaotic Structure of Weak Solutions of the Liouville Equation}
We observe in this section that for a large class of initial data, global-in-time mild solutions of \eqref{freetranny} enjoy a bijective correspondence with global-in-time mild solutions of the free transport equation on $T\mathbb{R}^{6}$ given by
\begin{equation}\label{classtrans}
\frac{\partial f}{\partial t}+(P\cdot\nabla_{Q})f=0 \quad \text{on}\hspace{2mm}T\mathbb{R}^{6},\tag{CT}
\end{equation}
where $f:T\mathbb{R}^{6}\times(-\infty, \infty)\rightarrow\mathbb{R}$. When the initial data under consideration are sufficiently smooth, mild solutions of \eqref{freetranny} give rise to global-in-time {\em classical} solutions of the above transport equation. 
\subsubsection{The Role of Free Transport}
To verify the above claim, we suppose first that $\mathfrak{f}$ is the global-in-time mild solution of \eqref{freetranny} associated to the initial datum
\begin{equation}\label{derpz}
\mathfrak{f}_{0}(\zeta):=\left\{
\begin{array}{ll}
F_{0}(Q, P) & \quad \text{if}\hspace{2mm}\zeta=[Q, P]_{1}\hspace{2mm}\text{for some}\hspace{2mm}[Q, P]\in T\mathbb{R}^{6}, \vspace{2mm}\\
G_{0}(Q, P) & \quad \text{if}\hspace{2mm}\zeta=[Q, P]_{2}\hspace{2mm}\text{for some}\hspace{2mm}[Q, P]\in T\mathbb{R}^{6},
\end{array}
\right.
\end{equation}
where $F_{0}\in C^{1}_{c}(T\mathbb{R}^{6})$ with $\mathrm{supp}\,F_{0}\subseteq\mathcal{D}$, and $G_{0}(Q, P):=F_{0}(\sigma_{\ast}(Q, P)Q, \sigma_{\ast}(Q, P)P)$ as above. Using arguments similar to those employed above, it is straightforward to verify that there exists a map $\mathfrak{S}_{0}:T\mathbb{R}^{6}\sqcup T\mathbb{R}^{6}\rightarrow\mathcal{P}\sqcup\mathcal{P}\times\mathbb{R}^{6}$ satisfying
\begin{equation}\label{sid}
S_{t}\#\mathfrak{f}_{0}=\mathfrak{S}_{0}\left(U_{t}\#\left(
\begin{array}{c}
F_{0}\\
G_{0}
\end{array}
\right)\right),
\end{equation}
where $\{U_{t}\}_{t\in\mathbb{R}}$ is the flow on the disjoint union $T\mathbb{R}^{6}\sqcup T\mathbb{R}^{6}$ given by
\begin{equation*}
U_{t}\left(
\begin{array}{c}
[X_{1}, V_{1}] \\

[X_{2}, V_{2}]
\end{array}
\right):=\left(
\begin{array}{c}
[X_{1}+tV_{1}, V_{1}]\\

[X_{2}+tV_{2}, V_{2}]
\end{array}
\right).
\end{equation*}
As such, studying mild solutions of \eqref{freetranny} on $\mathcal{P}\sqcup\mathcal{P}\times\mathbb{R}^{6}$ is equivalent to studying the free transport equation \eqref{classtrans} on 2 copies of $T\mathbb{R}^{6}$. Moreover, when initial data $\mathfrak{f}_{0}$ are of the form \eqref{derpz}, studying mild solutions \eqref{freetranny} on $\mathcal{P}\sqcup\mathcal{P}\times\mathbb{R}^{6}$ is equivalent to studying classical solutions of the free transport equation on a {\em single} copy of $T\mathbb{R}^{6}$. Indeed, one notes that if $F_{0}\in C^{1}_{c}(T\mathbb{R}^{6})$ and $G_{0}$ is the associated `twist' map, then both
\begin{equation*}
f(Q, P, t):=F_{0}(Q-tP, P)
\end{equation*}
and
\begin{equation*}
g(Q, P, t):=F_{0}(\sigma(Q, P)(Q-tP), \sigma(Q, P)P)
\end{equation*}
are global-in-time classical solutions of the free transport equation \eqref{classtrans} above.
\subsubsection{Corruption of Chaos}
Let us now discuss how identity \eqref{sid} reveals the manner in which the Liouville equation for hard spheres corrupts the structure of chaotic initial data. Suppose $F_{0}\in C^{1}_{c}(T\mathbb{R}^{6})$ is such that $\mathrm{supp}\,F_{0}\subseteq\mathcal{D}$. By basic algebraic properties of the Boltzmann scattering map, it follows that $G_{0}:=F_{0}\circ \Sigma_{\ast}$ also belongs to the same class of maps. Owing to the above observations, we write $\mathfrak{S}$ to denote the linear operator defined on the space of maps 
\begin{equation*}
\left\{
t\mapsto U_{t}\#\left(
\begin{array}{c}
F_{0}\\
F_{0}\circ\Sigma_{\ast}
\end{array}
\right)\,:\,F_{0}\in C^{1}_{c}(T\mathbb{R}^{6}), \hspace{2mm}\mathrm{supp}\,F_{0}\subseteq\mathcal{D}
\right\}
\end{equation*}
which maps the unique global-in-time classical solution $f$ of \eqref{classtrans} subject to the initial datum $F_{0}$ to the associated mild solution $\mathfrak{f}$ of \eqref{freetranny} subject to the initial datum \eqref{derpz}. It is well known that if $F_{0}\in C^{1}_{c}(T\mathbb{R}^{6})$ is taken to be of the shape
\begin{equation*}
F_{0}(Q, P):=\phi_{0}(q, p)\phi_{0}(\ov{q}, \ov{p})
\end{equation*}
for some $\phi_{0}\in C^{1}_{c}(T\mathbb{R}^{3})$, then the unique solution of \eqref{classtrans} is of the form
\begin{equation*}
f(Q, P, t)=\phi_{0}(q-tp, p)\phi_{0}(\ov{q}-t\ov{p}, \ov{p})
\end{equation*}
for all $(Q, P)\in T\mathbb{R}^{6}$ and $t\in (-\infty, \infty)$. In other words, it is a {\em chaotic} map on $T\mathbb{R}^{6}\times (-\infty, \infty)$. Finally, using the previously-discovered relationship 
\begin{equation*}
F=R'\mathfrak{f}\quad \text{in}\hspace{2mm}\mathscr{S}(T\mathbb{R}^{6})'
\end{equation*}
between $F$ and $\mathfrak{f}$, we have that the unique global-in-time weak solution of the Liouville equation admits the representation formula
\begin{equation}\label{chaosbut}
F=\mathsf{R}(f\otimes f),
\end{equation}
where $\mathsf{R}$ denotes the linear operator $\mathsf{R}:=R'\circ \mathfrak{S}$, which proves Theorem \ref{proppythm} above. One interpretation of the formula \eqref{chaosbut} is that while chaos is in general not propagated by the Liouville equation at the level of functions on $\mathcal{D}$, it is {\em essentially} propagated in a space of distributions by way of identity \eqref{chaosbut}. Note that nothing more than reflection and identification arguments were used to establish this identity.
\section{The BBGKY Hierarchy}\label{bbgkysection}
Our interest in solutions of the Liouville equation associated to hard sphere dynamics stems primarily from the BBGKY hierarchy to which it is connected. The derivation of global-in-time weak solutions of the BBGKY hierarchy which governs the marginals $F^{(1)}$ and $F^{(2)}$ of $F$ given by
\begin{equation*}
F^{(1)}(x, v, t):=\int_{\mathbb{R}^{3}\setminus B_{\varepsilon}(x)}\int_{\mathbb{R}^{3}}F(X, V, t)\,d\ov{v}d\ov{x}
\end{equation*}
for $[x, v]\in T\mathbb{R}^{3}$ and
\begin{equation*}
F^{(2)}(X, V, t):=F(X, V, t)
\end{equation*}
for $[X, V]\in\mathcal{D}$ becomes particularly elegant when viewed in the light of the approach of this article. Indeed, global-in-time weak solutions of the hierarchy are obtained simply from making a suitable choice of test function $\Phi$ in the definition of global-in-time weak solution of the Liouville equation. 
\subsection{Derivation of Weak Solutions of the BBGKY Hierarchy}
We choose a symmetric initial datum $F_{0}\in C^{1}_{c}(T\mathbb{R}^{6})$ with the property that $\mathrm{supp}\,F_{0}\subseteq\mathcal{D}$. By Theorem \ref{existandunique} above, we know that there exists a unique physical global-in-time weak solution $F\in C^{0}((-\infty, \infty), L^{1}(\mathcal{D}))$ of the Liouville equation \eqref{weakdude}. As such, the equation of the BBGKY hierarchy for the second marginal is immediately satisfied by $F^{(2)}=F$. It remains to show that the first marginal $F^{(1)}$ of $F$ satisfies the claimed equation. Indeed, using the observation that $F$ has compact support in $\mathcal{D}$ for all times, one deduces that $F$ satisfies
\begin{equation}\label{tellme}
\begin{array}{c}
\displaystyle \int_{\mathcal{P}}\int_{\mathbb{R}^{6}}\int_{-\infty}^{\infty}\left(\partial_{t}\Phi(X, V, t)+V\cdot\nabla_{X}\Phi(X, V, t)\right)F(X, V, t)\,dtdVdX \vspace{2mm}\\
=-\displaystyle\int_{\partial\mathcal{P}}\int_{\mathbb{R}^{6}}\int_{-\infty}^{\infty}\Phi(Y, V, t)F(Y, V, t)V\cdot\widetilde{\nu}(Y)\,dtdVdY
\end{array}
\end{equation}
for all $\Phi\in C^{\infty}(T\mathbb{R}^{6})$. By choosing the test function $\Phi\in C^{\infty}(T\mathbb{R}^{6})$ to be {\em independent} of its barred argument, namely
\begin{equation*}
\Phi(x, v, \ov{x}, \ov{v}, t)=\psi(x, v, t) \quad \text{for all}\hspace{2mm}[X, V]\in\mathcal{P}
\end{equation*} 
for some $\psi\in C^{\infty}(T\mathbb{R}^{3})$, it follows that the left-hand side of identity \eqref{tellme} becomes
\begin{equation*}
\begin{array}{c}
\displaystyle \int_{\mathcal{P}}\int_{\mathbb{R}^{6}}\int_{-\infty}^{\infty}\left(
\partial_{t}+V\cdot\nabla_{X}
\right)\Phi(X, V, t)F(X, V, t)\,dtdVdX\vspace{2mm}\\
\displaystyle = \int_{\mathcal{P}}\int_{\mathbb{R}^{6}}\int_{-\infty}^{\infty}\left(\partial_{t}+v\cdot\nabla_{x}\right)\psi(x, v, t)F(X, V, t)\,dtdVdX\vspace{2mm}\\
=\displaystyle \int_{\mathbb{R}^{3}}\int_{\mathbb{R}^{3}}\int_{-\infty}^{\infty}\left(\partial_{t}+v\cdot\nabla_{x}\right)\psi(x, v, t)\underbrace{\left(
\int_{\mathbb{R}^{3}\setminus B_{\varepsilon}(x)}\int_{\mathbb{R}^{3}}F(X, V, t)\,d\ov{x}d\ov{v}
\right)}_{F^{(1)}(z, t)=}\,dtdvdx.
\end{array}
\end{equation*}
Similarly, by changing measure
\begin{equation*}
V\mapsto\left(I-\left[
\begin{array}{c}
n \\
-n
\end{array}
\right]\otimes \left[
\begin{array}{c}
n \\
-n
\end{array}
\right]\right)V
\end{equation*}
on $\mathcal{C}^{-}(n)$ for each $n\in\mathbb{S}^{2}$ and using the fact that $(v_{n}'-\ov{v}_{n}')\cdot n=-(v-\ov{v})\cdot n$, where
\begin{equation*}
\begin{array}{c}
v_{n}'=v-((v-\ov{v})\cdot n)n, \vspace{2mm}\\
\ov{v}_{n}'=\ov{v}+((v-\ov{v})\cdot n)n,
\end{array}
\end{equation*} 
the right-hand side of identity \eqref{tellme} becomes
\begin{align*}
\int_{\partial\mathcal{P}}\int_{\mathbb{R}^{6}}\int_{-\infty}^{\infty}\Phi(Y, V, t)F(Y, V, t)V\cdot\widetilde{\nu}(Y)\,dtdVdY \vspace{2mm}\\
=\frac{1}{\sqrt{2}}\int_{\mathbb{R}^{3}}\int_{\mathbb{S}^{2}}\int_{\mathbb{R}^{6}}\int_{-\infty}^{\infty}\psi(y, v, t)F^{(2)}(y, v, y+\varepsilon n, \ov{v}, t)(v-\ov{v})\cdot n(y)\,dtdVdndx\vspace{2mm}\\
=\frac{1}{\sqrt{2}}\int_{\mathbb{R}^{3}}\int_{\mathbb{S}^{2}}\int_{\mathcal{C}^{-}(n)}\int_{-\infty}^{\infty}\psi(y, v, t)F^{(2)}(y, v, y+\varepsilon n, \ov{v}, t)(v-\ov{v})\cdot n\,dtdVdndx \vspace{2mm}\\
+\frac{1}{\sqrt{2}}\int_{\mathbb{R}^{3}}\int_{\mathbb{S}^{2}}\int_{\mathcal{C}^{+}(n)}\int_{-\infty}^{\infty}\psi(y, v, t)F^{(2)}(y, v, y+\varepsilon n, \ov{v}, t)(v-\ov{v})\cdot n\,dtdVdndx\vspace{2mm}\\
=\frac{1}{\sqrt{2}}\int_{\mathbb{R}^{3}}\int_{\mathbb{S}^{2}}\int_{\mathcal{C}^{-}(n)}\int_{-\infty}^{\infty}\psi(y, v, t)\bigg(F^{(2)}(y, v, y+\varepsilon n, \ov{v}, t) \vspace{2mm}\\
-F^{(2)}(y, v_{n}', y+\varepsilon n, \ov{v}_{n}', t)\bigg)(v-\ov{v})\cdot, n\,dtdVdndx.
\end{align*}
This demonstrates that the maps $F^{(1)}$ and $F^{(2)}$ constitute a global-in-time weak solution of the BBGKY hierarchy.
%\section{Final Remarks}
\subsection*{Acknowledgement}
The author would like to extend his sincere thanks to Professor Piero Marcati for his kind invitation to visit the Gran Sasso Science Institute in L'Aquila, Italy, where part of this work was done. 
\bibliography{bib}

\end{document}